\documentclass[reqno,10pt,letterpaper]{amsart}
\usepackage{amsmath,amssymb,amsthm,amscd,graphicx,mathrsfs,url,dsfont,enumitem}
\usepackage[usenames,dvipsnames]{color}
\usepackage[colorlinks=true,linkcolor=Blue,citecolor=Green,urlcolor=Blue]{hyperref}
\usepackage{dsfont}
\usepackage[super]{nth}
\usepackage[open, openlevel=2, depth=3, atend]{bookmark}
\hypersetup{pdfstartview=XYZ}
\usepackage{epigraph}
\usepackage{mathtools}
\usepackage{enumitem}
\usepackage[mathscr]{eucal}
\usepackage{comment}

\makeatletter
\@namedef{subjclassname@2020}{\textup{2020} Mathematics Subject Classification}
\makeatother

\let\oldtocsection=\tocsection
\let\oldtocsubsection=\tocsubsection
\let\oldtocsubsubsection=\tocsubsubsection
\renewcommand{\tocsection}[2]{\hspace{0em}\oldtocsection{#1}{#2}}
\renewcommand{\tocsubsection}[2]{\hspace{1em}\oldtocsubsection{#1}{#2}}
\renewcommand{\tocsubsubsection}[2]{\hspace{2em}\oldtocsubsubsection{#1}{#2}}

\def\?[#1]{\textbf{[#1]}\marginpar{\Large{\textbf{??}}}}

\renewcommand{\tilde}{\widetilde}          
\DeclareMathSymbol{\leqslant}{\mathalpha}{AMSa}{"36} 
\DeclareMathSymbol{\geqslant}{\mathalpha}{AMSa}{"3E} 
\DeclareMathSymbol{\eset}{\mathalpha}{AMSb}{"3F}     
\renewcommand{\leq}{\;\leqslant\;}                   
\renewcommand{\geq}{\;\geqslant\;}                   
\renewcommand{\d}{\mathrm{d}}             
\renewcommand{\emptyset}{\eset}

\numberwithin{equation}{section}

\newtheorem{theorem}{Theorem}[section]
\newtheorem{lemma}[theorem]{Lemma}
\newtheorem{proposition}[theorem]{Proposition}

\theoremstyle{remark}
\newtheorem{remark}{Remark}

\newtheorem{open}{Open question}
\theoremstyle{definition}

\newcommand{\C}{\mathbb{C}}
\newcommand{\D}{\mathbb{D}}
\newcommand{\R}{\mathbb{R}}
\newcommand{\Z}{\mathbb{Z}}
\renewcommand{\H}{\mathbb{H}}
\newcommand{\N}{\mathbb{N}}

\newcommand{\E}{\mathbb{E}}
\renewcommand{\P}{\mathbb{P}}
\renewcommand{\S}{\mathbb{S}}

\renewcommand{\Im}{\mathrm{Im}}
\renewcommand{\Re}{\mathrm{Re}}

\newcommand{\cT}{\mathcal{T}}

\newcommand{\cE}{\mathcal{E}}
\newcommand{\cC}{\mathcal{C}}
\newcommand{\cH}{\mathcal{H}}
\newcommand{\cF}{\mathcal{F}}

\newcommand{\cN}{\mathcal{N}}

\newcommand{\cV}{\mathcal{V}}

\newcommand{\cP}{\mathcal{P}}
\newcommand{\ind}{\mathds{1}}

\newcommand{\cW}{\mathcal{W}}
\newcommand{\cD}{\mathcal{D}}

\newcommand{\laweq}{\overset{\text{law}}{=}}
\newcommand{\cI}{\mathcal{I}}

\newcommand{\del}{\partial}

\newcommand{\bS}{\mathbf{S}}

\newcommand{\cQ}{\mathcal{Q}}

\newcommand{\bA}{\mathbf{A}}
\newcommand{\bL}{\mathbf{L}}

\newcommand{\bH}{\mathbf{H}}

\newcommand{\A}{\mathbb{A}}

\newcommand{\boldw}{\boldsymbol{w}}
\newcommand{\bolds}{\boldsymbol{s}}
\newcommand{\I}{\mathbb{I}}
\newcommand{\boldx}{\boldsymbol{x}}
\newcommand{\cR}{\mathcal{R}}

\author{Guillaume Baverez}
\address{Beijing International Center for Mathematical Research, Peking University}
\email{guillaume.baverez@bicmr.pku.edu.cn}

\author{Baojun Wu}
\address{Beijing International Center for Mathematical Research, Peking University}
\email{wubaojunmathe@outlook.com}

\title[Level 2 HEM in Liouville CFT]{Higher equations of motion at level 2 in Liouville CFT}


\begin{document}

\maketitle

\begin{abstract}
We prove conjectures of Zamolodchikov and Belavin--Belavin in Liouville conformal field theory (CFT), which are generalisations of the celebrated Belavin--Polyakov--Zamolodchikov equations known as the higher equations of motion. Algebraically, these equations give examples of non-zero singular states in Virasoro modules, which is a relatively rare phenomenon in the physical study of CFT. In probability theory, these equations and their variants have been instrumental in the rigorous derivation of the structure constants of Liouville CFT in the unit disc.

The proof builds on a previous work of ours studying the analytic continuation of the Poisson operator of Liouville theory. The main novelty is that this operator admits poles on the Kac table, and the higher equations of motions are obtained via a residue computation.   
\end{abstract}

\tableofcontents

\section{Introduction}

	\subsection{Motivation and background}\label{subsec:overview}
	
In the last decade, a series of breakthrough results led to a rigorous construction of Liouville conformal field theory (CFT), using probabilistic techniques based on the Gaussian free field and Gaussian multiplicative chaos \cite{DKRV16,huang2018,DRV16_tori,Guillarmou2019}. Prior to this, the focus in physics was put on the algebraic structure of Liouville CFT (see e.g. \cite{Teschner_revisited,Nakayama,Ribault14} for reviews in physics)
rather than its analytic aspects (with some notable exceptions exploring the connections between Liouville theory and the complex-analytic formalism of CFT proposed by Friedan--Shenker \cite{FriedanShenker87}, see e.g. \cite{TakhtajanTeo_Liouville} and references therein). Based on the assumption of conformal symmetry and using results from representation theory, the algebraic approach has striking predictions (such as exact formulas for certain correlations and the ``conformal bootstrap hypothesis" \cite{DornOtto94,Zamolodchikov96,Teschner03,BPZ84}) which were verified in the probabilistic setting \cite{KRV_DOZZ,GKRV20_bootstrap,GKRV21_Segal}. 
The cornerstone of this integrable structure is a family of PDEs satisfied by correlation functions involving so-called degenerate fields: these are the celebrated BPZ equations \cite{BPZ84}, proved in \cite{KRV19_local} for the Liouville CFT. The existence of such PDEs is a consequence of conformal symmetry (in the form of Ward identities) and the physically motivated assumption that ``states with vanishing norm must vanish identically" (see \cite[Section 1.1]{BW} for a discussion of the meaning of this phrase). Still, the BPZ equations are down-the-road consequences of conformal symmetry, and more recent works in probability theory have focused on the construction of the whole algebraic structure expected for a CFT: a family of (unbounded) operators acting on the Hilbert space of the theory, which represents two commuting copies of the Virasoro algebra and exponentiate to a projective representation of Segal's semigroup of annuli \cite{GKRV20_bootstrap,BGKRV22,BGKRV24_semigroup}. The existence of such a structure constrains the Liouville Hilbert space through representation theory \cite{GKRV20_bootstrap,BGKRV22}, and the BPZ equations can be reformulated as a statement on the algebraic structure of certain representations. While this structure is usually taken as an axiom in physics, it needs to be proved in the probabilistic representation \cite[Theorem 1.2]{BW}.

This paper can be seen as a continuation of \cite{BW}, and is concerned with two cases where this basic axiom does not hold: these are known as the \emph{higher-equations of motion} (HEM) and have been predicted in \cite{Zamolodchikov03_HEM,Belavin2_bHEM}. Aside from the intrinsic interest of these equations (which have applications in minimal gravity and intersection theory \cite{BelavinZamolodchikov06_integrals,Artemev_2022,artemev202422p1minimalstringintersection}), they stress the importance of not taking the BPZ equations for granted. In the probabilistic literature, some hints for the HEMs appeared in \cite{Ang23_zipper}, where the (level 2) BPZ equations for Liouville CFT on the disc are proved under some conditions on the cosmological constants. Relaxing this condition, Cercl\'e proves the full HEMs for correlation functions in the disc, in an independent work which appeared online a few weeks after the present one. The (level 2) HEMs as stated in \cite{Belavin2_bHEM} (i.e. at the level of states rather than correlations) is the topic of Theorem \ref{thm:boundary_hem} below. As noted by both Cercl\'e and us, these HEMs do not hold for some range of the parameters (second item of Theorem \ref{thm:boundary_hem}, case $\gamma>\sqrt{2}$). The special case proved by Ang has been instrumental in a wide programme aiming at the computation of all structure constants of boundary Liouville CFT \cite{ARSZ23}, following predictions from physics \cite{FZZ00,PonsotTeschner_correl,Hosomichi}. In addition to the boundary theory, Theorem \ref{thm:bulk} studies the bulk version of these HEMs, which were predicted earlier than the boundary versions by Zamolodchikov \cite{Zamolodchikov03_HEM} (and are not addressed in \cite{Ang23_zipper,Cercle_BPZ}). Finally, we mention that Cercl\'e--Huguenin recently announced a version of the HEMs for the boundary $\mathfrak{sl}_3$-Toda CFT~\cite{Cercle_hem_toda}.

In the remainder of this introduction, we state our main results and provide an outline of the proofs, as well as some open questions. The proofs of Theorem \ref{thm:bulk} and Theorem \ref{thm:boundary_hem} follow the same strategy and can be read independently (with quite some extra technicalities for the bulk version). In particular, the reader only interested in the boundary HEM may skip directly to Section \ref{sec:boundary} only.

	\subsection{Main results}
Here and in the sequel, $\gamma$ and $\mu$ are parameters satisfying
\[\gamma\in(0,2);\qquad\mu>0.\]
We refer to $\mu$ as the \emph{cosmological constant}. We also set,
\[Q:=\frac{\gamma}{2}+\frac{2}{\gamma}>2;\qquad c_\mathrm{L}:=1+6Q^2>25.\]

The \emph{Kac tables} are the discrete sets
\[kac^\pm:=\left\lbrace(1\pm r)\frac{\gamma}{2}+(1\pm s)\frac{2}{\gamma}\big|\,r,s\in\N^*\right\rbrace,\qquad kac:=kac^+\sqcup kac^-,\]
where $\N^*=\Z_{>0}$. Note that $kac^+=2Q-kac^-$, and $kac^-\subset(-\infty,0]$. We use the notation $\alpha_{r,s}:=(1-r)\frac{\gamma}{2}+(1-s)\frac{2}{\gamma}$
for all $r,s\in\Z$, and note that:
\[\alpha_{1,2}=-\frac{2}{\gamma}\qquad\text{and}\qquad\alpha_{2,1}=-\frac{\gamma}{2}.\]
 For all $\alpha\in\C$, we define the \emph{conformal weight}
\begin{equation}\label{eq:def_Delta}
\Delta_\alpha:=\frac{\alpha}{2}\left(Q-\frac{\alpha}{2}\right).
\end{equation}
Observe that $\Delta_\alpha=\Delta_{2Q-\alpha}$.		
	
		\subsubsection{Bulk HEM}	
In order to state our main result, we need to introduce some notions of Liouville CFT, which are recalled in greater detail in Section \ref{subsec:bulk_setup}. 

The Hilbert space of Liouville CFT $\cH$ is the $L^2$-space of a Gaussian measure with values in the space of distributions on the unit circle. It comes with a Hilbert space $\cH$ and a family of densely-defined, unbounded operators $(\bL_n,\tilde{\bL}_n)_{n\in\Z}$ on $\cH$, which represent two commuting families of the Virasoro algebra in a suitable sense. The Hamiltonian is the operator $\bH=\bL_0+\tilde{\bL}_0$, which is self-adjoint and induces a positive definite quadratic form with dense domain $\cD(\cQ)\subset\cH$. Then, for all $\beta>0$ and all $n\in\Z$, the Virasoro generators are bounded as operators $\bL_n,\tilde{\bL}_n:e^{-\beta c}\cD(\cQ)\to e^{-\beta c}\cD'(\cQ)$, where $\cD'(\cQ)$ is the continuous dual to $\cD(\cQ)$ (see Section \ref{subsec:bulk_setup} for a precise definition of the weighting). Furthermore, for each pair of Young diagrams $\nu,\tilde{\nu}$ (see Section~\ref{subsec:bulk_setup} for the definition), there exists an analytic family of generalised eigenstates $\alpha\mapsto\Psi_{\alpha,\nu,\tilde{\nu}}\in e^{-\beta c}\cD(\cQ)$ (with $\beta>|Q-\Re(\alpha)|$), such that $\Psi_{\alpha,\nu,\tilde{\nu}}$ is obtained by acting on $\Psi_\alpha$ with the Virasoro generators. In particular, these operators can be viewed as endomorphisms of the vector space $\cW_\alpha$ (algebraically) spanned by the generalised eigenstates, which is the point of view adopted in the next statement (see Section \ref{subsec:bulk_setup} and \cite[Section 4.4]{BGKRV22} for some details on these properties). Since $\Psi_\alpha$ is analytic with values in $e^{-\beta c}\cD(\cQ)$ for some $\beta$, we can consider its complex derivative $\Psi_\alpha':=\del_\alpha\Psi_\alpha\in e^{-\beta c}\cD(\cQ)$. In the next statement, $\beta$ is any fixed number satisfying the above property.

\begin{theorem}[Bulk HEM]\label{thm:bulk}
 For all $\alpha\in\C$, define
\[\bS_\alpha:=\alpha^2\bL_{-2}+\bL_{-1}^2;\qquad\tilde{\bS}_\alpha:=\alpha^2\tilde{\bL}_{-2}+\tilde{\bL}_{-1}^2.\]
\begin{enumerate}[label={\arabic*.}]
\item The following equality holds in $e^{-\beta c}\cD(\cQ)$:
\begin{equation}\label{eq:bulk_(1,2)}
\bS_{\alpha_{1,2}}\tilde{\bS}_{\alpha_{1,2}}\Psi_{\alpha_{1,2}}'=\pi\mu\frac{8}{\gamma^3}\left(1-\frac{\gamma^2}{4}\right)^2\Psi_{\alpha_{-1,2}}.
\end{equation}
\item The following equality holds in $e^{-\beta c}\cD(\cQ)$:
\begin{equation}\label{eq:bulk_(2,1)}
\bS_{\alpha_{2,1}}\tilde{\bS}_{\alpha_{2,1}}\Psi_{\alpha_{2,1}}'=\left\lbrace\begin{aligned}
&-\frac{\gamma^5}{32}\left(\pi\mu\frac{\Gamma(\frac{\gamma^2}{4})}{\Gamma(1-\frac{\gamma^2}{4})}\right)^2\frac{\Gamma(1-\frac{\gamma^2}{2})}{\Gamma(\frac{\gamma^2}{2})}\Psi_{\alpha_{-2,1}},&\text{if }\gamma<\sqrt{2};\\
&0,&\text{if }\gamma>\sqrt{2}.
\end{aligned}\right.
\end{equation}
\end{enumerate}
\end{theorem}
 Theorem \ref{thm:bulk} confirms the prediction of Zamolodchikov, except the $(2,1)$-equation in the regime $\gamma\in(\sqrt{2},2)$ (where $\alpha_{-2,1}>Q$). The reason is a freezing phenomenon which makes the equation break down (see Open question~\ref{open:supercritical}). In the other cases, the scalar multiple of the primary field coincides with \cite[Equation (5.17)]{Zamolodchikov03_HEM}\footnote{With the dictionary $(r,s)\leftrightarrow(n,m)$ and $\frac{\gamma}{2}\leftrightarrow b$, and \cite{Zamolodchikov03_HEM} uses the notation $\gamma(z)=\frac{\Gamma(z)}{\Gamma(1-z)}$.}. Equation \eqref{eq:bulk_(2,1)} has a non-trivial limit as $\gamma\to\sqrt{2}$ from below. We expect that the $(2,1)$-equation holds in this case, but it requires a more detailed analysis of Gaussian multiplicative chaos (see Open question~\ref{open:critical}).

  We will prove Theorem \ref{thm:bulk} in Section \ref{sec:bulk}. As in \cite{BW}, the strategy is to study the analytic continuation of the Poisson operator. The main difference is that the extension has a simple pole at $\alpha_{1,2}$ or $\alpha_{2,1}$. The RHS of \eqref{eq:bulk_(2,1)} is then identified with the residue of the Poisson operator. At the time of writing, we are aware of an online manuscript studying the $(2,1)$-HEM for correlation functions on the Riemann sphere \cite{AR_hem}. Their proof is close in spirit to Cercl\'e's proof of the boundary HEM, but the scalar multiple of the primary field of \cite{AR_hem} differs from ours (and Zamolodchikov's); this seems to be due to the omission of a pole in the study of the correlation function of the singular state.

Theorem \ref{thm:bulk} is an equality of \emph{states} in Liouville CFT. The link between states and correlation functions in Liouville CFT is given by the so-called \emph{amplitudes} introduced in \cite{GKRV21_Segal}: a correlation function on an arbitrary compact Riemann surface can be obtained as the inner-product of a state with the amplitude of a surface with boundary. On the other hand, \cite{BGKRV22,BGKRV24_semigroup} relate the action of the Virasoro algebra to deformations of the complex structure. In a work in progress of the first author and Guillarmou, Kupiainen, Rhodes \cite{BGKRV_blocks}, it is shown in full generality how to relate the Virasoro algebra to deformations of the complex structure on an arbitrary compact Riemann surface. Concretely, the machinery developed in \cite{BGKRV_blocks} implies that conformal blocks involving Virasoro descendants are obtained from conformal blocks of primary fields by applying a partial differential operator on the Teichm\"uller space of the surface. In particular, Theorem \ref{thm:bulk} implies that conformal blocks involving the state $\bS_{\alpha_{r,s}}\tilde{\bS}_{\alpha_{r,s}}\Psi_{\alpha_{r,s}}'$ (for $rs=2$) satisfy a certain PDE on Teichm\"uller space. This generalises the celebrated BPZ equations. Since this relies on the work in progress \cite{BGKRV_blocks}, we refrain from writing a full statement for the time being.

		\subsubsection{Boundary HEMs}
Liouville CFT has a boundary counterpart, which has been under intense investigation in the probability community over the past few years. On the one hand, there is a vast integrability programme which culminated in the derivation of all structure constants of the theory \cite{ARSZ23}. These formulas are the analogues of the celebrated DOZZ formula of bulk LCFT \cite{KRV_DOZZ}. On the other hand, there is work in progress studying the conformal bootstrap for boundary LCFT \cite{GRVW}. Our results could be formulated compactly (as in Theorem \ref{thm:bulk}) using the technology of a work still in progress, but we will instead follow a more pedestrian route until this work appears online.

The details of boundary LCFT are recalled in Section \ref{subsec:boundary_setup}. The Hilbert space $\cH$ is the $L^2$-space of a Gaussian measure with values in the space of distributions on the \emph{half}-circle $\S^1\cap\H$. In addition to the bulk cosmological constant $\mu$, it depends on two additional constants $\mu_\mathrm{L},\mu_\mathrm{R}>0$ associated with the intervals $\I^-=(-1,0)$ and $\I^+=(0,1)$ respectively. The Sugawara construction $(\bL_n^{0,\alpha})_{n\in\Z}$ is a family of densely defined, unbounded operators which represent the Virasoro algebra. We define $\bS^0_\alpha:=\alpha^2\bL_{-2}^{0,\alpha}+(\bL_{-1}^{0,\alpha})^2$. 

As in the bulk theory, one can define a family of states $\alpha\mapsto\Psi_\alpha^\del\in e^{-\beta c}\cH$ depending analytically on $\alpha$ and with values in the weighted Hilbert space $e^{-\beta c}\cH$ with $\beta>|Q-\Re(\alpha)|$ (see Section \ref{subsec:boundary_setup} for the precise definition). In Section \ref{subsec:boundary_setup}, we introduce a map which informally speaking sends the ``free field state" $\bS^0_\alpha\ind$ to the ``Liouville state" $\Phi_\alpha(\bS^0_{2Q-\alpha}\ind)\in e^{-\beta c}\cH$. Analogously to \cite[Theorem 4.5]{BGKRV22}, it will be shown in a work in progress \cite{GRW2} that this state is the descendant of a primary field $\Psi_\alpha^\del$ for some Virasoro representation $(\bL_n)_{n\in\Z}$. In the meantime, we introduce this state $\Psi_\alpha^\del$ in Section \ref{subsec:boundary_setup}, and show its analyticity for $\alpha$ in a complex neighbourhood of the half-line $(-\infty,Q-\epsilon)$ for each $\epsilon>0$ (see Lemma \ref{lem:analytic}). In the next statement, $\beta$ is any number satisfying the above condition.

\begin{theorem}[Boundary HEM]\label{thm:boundary_hem}
The expression $\Phi_\alpha(\bS^0_{2Q-\alpha}\ind)\in e^{-\beta c}\cH$ admits an analytic extension which is regular in a neighbourhood of $\alpha_{1,2}$ and $\alpha_{2,1}$. Moreover,
\begin{enumerate}[label={\arabic*.}]
\item The following equality holds in $e^{-\beta c}\cH$:
\begin{equation}\label{eq:boundary_(1,2)}
\Phi_{\alpha_{1,2}}(\bS_{2Q-\alpha_{1,2}}^0\ind)=\frac{4}{\gamma^2}\left(1-\frac{\gamma^2}{4}\right)(\mu_\mathrm{L}+\mu_\mathrm{R})\Psi_{\alpha_{-1,2}}^\del.
\end{equation}
\item The following equality holds in $e^{-\beta c}\cH$:
\begin{equation}\label{eq:boundary_(2,1)}
\Phi_{\alpha_{2,1}}(\bS_{2Q-\alpha_{2,1}}^0\ind)=\left\lbrace\begin{aligned}
&\frac{\gamma^2}{4}\left(\mu_\mathrm{L}^2-2\mu_\mathrm{L}\mu_\mathrm{R}\cos(\pi\frac{\gamma^2}{4})+\mu_\mathrm{R}^2-\mu\sin(\pi\frac{\gamma^2}{4})\right)\frac{\Gamma(\frac{\gamma^2}{4})\Gamma(1-\frac{\gamma^2}{2})}{\Gamma(1-\frac{\gamma^2}{4})}\Psi_{\alpha_{-2,1}}^\del&\text{ if }\gamma<\sqrt{2};\\
&0&\text{if }\gamma>\sqrt{2}.
\end{aligned}\right.
\end{equation}
\end{enumerate}
\end{theorem}

By scaling, we can assume $\mu=1$. Then, the prefactor is a second-degree polynomial in $\mu_\mathrm{L},\mu_\mathrm{R}$, and the RHS of \eqref{eq:boundary_(2,1)} vanishes on its zero set. The zero set is just an algebraic curve of degree 2, i.e. a conic section. This conic section was first exhibited in \cite{FZZ00} and we refer to it as the \emph{FZZ conic section}. On the FZZ conic section, the RHS of \eqref{eq:boundary_(2,1)} vanishes. For general $(r,s)\in\N^*$, the prefactor is expected to be a polynomial in $\mu_\mathrm{L},\mu_\mathrm{R}$ of degree $r$, so that the BPZ equation should hold on a certain algebraic curve of degree $r$. This algebraic curve factorises explicitly \cite[Equation (2.35)]{Belavin2_bHEM}. 

Following the same mechanism as in the bulk case, Theorem \ref{thm:boundary_hem} should lead to PDEs for correlation functions on arbitrary Riemann surfaces with boundaries and punctures. This will follow from a series of works on boundary Liouville theory \cite{GRVW,GRW2}. As mentioned in Section \ref{subsec:overview}, \cite{Ang23_zipper} proves the BPZ equations at the level of correlation functions of Liouville CFT in the disc, which corresponds to the cases where the right-hand-sides in Theorem \ref{thm:boundary_hem} vanish. The proof of \cite{Ang23_zipper} is quite different from ours since it uses the mating-of-trees theory \cite{MatingOfTrees}: it relies on the clever observation that the BPZ differential operator is the generator of a certain Schramm--Loewner evolution. Using this observation, the BPZ equation is a consequence of the fact that the correlation function evolves as a martingale under the dynamics generated by this operator. In this context, the quantities $\cos(\pi\frac{\gamma^2}{4})$ and $\sin(\pi\frac{\gamma^2}{4})$ appearing in the right-hand-side of the $(2,1)$-HEM have a special meaning: they are respectively the correlation and the speed of the Brownian motions corresponding to the evolution of the two quantum lengths (parametrised by quantum area). In our proof, these quantities are related to residues of certain Selberg integrals on the degenerate weights (Lemma \ref{lem:ff_residues}). 

Finally, \cite{Cercle_BPZ} also works at the level of correlation functions, but proves all HEMs in a way that matches Theorem \eqref{thm:boundary_hem}. The method is closer to earlier works in the field, using the conformal Ward identities. Cercl\'e introduces a ``descendant field at level 2" which can be thought of as the descendant state $\Phi_\alpha(\bS^0_{2Q-\alpha}\ind)$ from this paper. The insertion of this state combined with Ward identities leads to the BPZ equation. In our approach, the Ward identities will be replaced with the conformal blocks machinery that will be developed thanks to the boundary version of the conformal bootstrap. Indeed, this machinery gives a way to streamline the Ward identities without computations and independently of the geometric setup.

	  \subsection{Outline and open questions}
The strategy for the proof of Theorems \ref{thm:bulk} and \ref{thm:boundary_hem} is the following.
\begin{enumerate}[label={\arabic*.}]
\item We express the singular state as certain singular integral of the correlation function (Sections \ref{subsec:expression_bulk} and \ref{subsec:expression_boundary}). These expressions are reminiscent of those found in \cite{BW}, namely they look like primary fields with additional $\gamma$-insertions integrated over the disc (see \eqref{eq:def_integrals} and \eqref{eq:def_integrals_boundary}).
\item Contrary to \cite{BW}, these integrals have a simple pole at $\alpha_{1,2}$, $\alpha_{2,1}$. The presence of this pole is the crux of the HEM: the RHS of the HEMs are given by the residues of these integrals. To compute this residue, we rely on fusion estimates for correlation functions (Propositions \ref{prop:fusion} and \ref{prop:fusion_boundary}), and exact formulas for Dotsenko--Fateev/Selberg integrals.
\item The theory developed in \cite{GKRV20_bootstrap,BGKRV22} tells us that the Poisson operator is analytic away from the Kac table. Since $\alpha_{1,2}<\alpha_{2,1}$, the pole at $\alpha_{1,2}$ is easier to study. To study the pole at $\alpha_{2,1}$, one needs to find a probabilistic expression for the analytic continuation beyond the pole at $\alpha_{1,2}$. This done in Propositions \ref{prop:mero} and \ref{prop:mero_boundary}, using a method similar to the one in \cite[Section 3.3]{BW}.
\end{enumerate}

Our results leave some questions unanswered, which we hope to address in future work.

\begin{open}[The critical $(2,1)$-HEM] \label{open:critical}
The RHS of \eqref{eq:bulk_(2,1)} has a limit as $\gamma\to\sqrt{2}$ from below, but this limit is deceptively simple: the constant prefactor has a pole (as a function of $\gamma$), which is compensated by the zero of the primary state. Thus, the limit involves the derivative state $\Psi_Q'=\underset{\alpha\nearrow Q}\lim\,\frac{\Psi_\alpha}{\alpha-Q}$ corresponding to the ``$Q$-vertex operator" studied in \cite{DKRV17_Seiberg}. The estimates used in this paper are not sufficient to prove the case $\gamma=\sqrt{2}$, but it should be possible to adapt some techniques from \cite{DKRV17_Seiberg,Baverez19,BaverezWong18} to get more precise estimates (although the analysis looks more subtle). We view it as an interesting question on Gaussian multiplicative chaos.
\end{open}

\begin{open}[The general $(r,s)$-HEM]
Can our techniques be adapted to general $(r,s)$? In \cite{BW}, we relied on an induction formula to express a singular integral in terms of ``lower" ones (with respect to a certain partial order on Young diagrams). The same technique should apply to the HEM, but it will be computationally heavier, with the additional difficulty of computing the residue. In the end, we expect that the only non-zero contribution at $\alpha_{r,s}$ is given by the residue of $\int_{\D^r}\Psi_\alpha(w_1,...,w_r)\prod_{j=1}^r\frac{|\d w_j|^2}{|w_j|^{2s}}$. We note that the $(1,s)$-equation is the simplest and could be directly deduced from the techniques of this paper, but we prefer to treat the general case all at once.
\end{open}

\begin{open}[``Analytic continuation"]\label{open:supercritical}
Due to a freezing phenomenon, the $(2,1)$-equation crashes down for $\gamma>\sqrt{2}$. However, even though $\alpha_{-2,1}>Q$ for $\gamma>\sqrt{2}$, the state $\Psi_{\alpha_{-2,1}}$ can be analytically continued by $R(\alpha_{-2,1})\Psi_{2Q-\alpha_{-2,1}}$ \cite{BGKRV22}. Is there a probabilistic model such that the $(2,1)$-equation is non-trivial and given by this analytic continuation? Based on \cite[Section 3.2]{AHS21_welding}, one possibility would be to look at Poisson collections of quantum spheres.
\end{open}

	\subsection*{Acknowledgements}
This project was initiated during the school ``Three facets of gravity" held in spring 2023 at Humboldt University, thanks to the support of Kolleg Mathematik Physik Berlin. We thank the anonymous referees for their detailed feedback, which led to many improvements of the manuscript. G.B. acknowledges support from the ANR 21-CE40-0003. B.W. \ was supported by National Key R\&D Program of China (No. 2023YFA1010700).

\section{Bulk HEM}\label{sec:bulk}
In this section, we prove Theorem \ref{thm:bulk}. First, we give some background and notations on Liouville CFT and the Poisson operator in Section \ref{subsec:bulk_setup}. In Section \ref{subsec:expression_bulk}, we give a probabilistic expression for the descendant states, which is valid up to $\alpha<\alpha_{1,2}$. The $(1,2)$-HEM is easily deduced from this expression in Section \ref{subsec:first_pole_bulk}. The $(2,1)$-equation is slightly harder since we first need to find a probabilistic expression for the analytic extension, valid up to $\alpha<\alpha_{2,1}$ (Proposition \ref{prop:mero}). With this expression in hand, the $(2,1)$-equation is deduced in a similar fashion.

	\subsection{Setup and background}\label{subsec:bulk_setup}
The content of this section is extracted from \cite{GKRV20_bootstrap,BGKRV22}, and follows closely the summary given in \cite[Section 2.2-3]{BW}. We also refer to \cite[Appendix B]{BW} for some elementary background on the Virasoro algebra.

		\subsubsection{Free field modules}
Let $\cF:=\C[(\varphi_n,\bar{\varphi}_n)_{n\geq1}]$ be the space of polynomials in countably many complex variables $\varphi_n$ and their complex conjugates $\bar{\varphi}_n$. The constant function $\ind$ is the \emph{vacuum vector}. A \emph{Young diagram of length $\ell\in\N$} is a non-increasing sequence of $\ell$ positive integers $\nu=(\nu_1,...\nu_\ell)$. We denote by $|\nu|=\sum_{l=1}^\ell\nu_l$ and call it the \emph{level of $\nu$}. We denote the set of all Young diagrams (resp. Young diagrams of level $N$) by $\cT$ (resp. $\cT_N$) and set $p(N):=\#\cT_N$. By convention, $p(0)=1$. 

Let $\P_{\S^1}$ be the law of the $\R$-valued, $\log$-correlated Gaussian field $\varphi$ on the unit circle:
\[\E[\varphi(e^{i\theta})\varphi(e^{i\theta'})]=\log\frac{1}{|e^{i\theta}-e^{i\theta'}|},\qquad\theta,\theta'\in\R.\]
Due to the logarithmic blow-up on the diagonal, this field only exists (almost surely) as a random distribution in the Sobolev space $H^{-s}(\S^1)$ for arbitrary small $s>0$. The expansion in Fourier modes reads
\begin{equation}\label{eq:fourier}
\varphi=\sum_{n\in\Z\setminus\{0\}}\varphi_ne_n,
\end{equation}
where $e_n(e^{i\theta})=e^{ni\theta}$ is the standard basis. We have $\varphi_{-n}=\bar{\varphi}_n$ for all $n\in\Z\setminus\{0\}$. Under $\P_{\S^1}$, the sequence $(\varphi_n)_{n\geq1}$ is made of independent complex Gaussians $\cN_\C(0,\frac{1}{2n})$. The harmonic extension of $\varphi$ to the unit disc is
\[P\varphi(z)=2\Re\left(\sum_{n=1}^\infty\varphi_nz^n\right).\]
The covariance kernel of $P\varphi$ is
\[\E[P\varphi(z)P\varphi(w)]=\log\frac{1}{|1-z\bar{w}|}=:G_\del(z,w),\qquad\forall z,w\in\D.\]

The space $\cF$ is a dense subspace of $L^2(\P_{\S^1})$. The \emph{Liouville Hilbert space} is 
\[\cH:=L^2(\d c\otimes\P_{\S^1}),\]
where $\d c$ is Lebesgue measure on $\R$. Samples of $\d c\otimes\P_{\S^1}$ are written $c+\varphi$, with $c$ being the zero mode of the field. Following \cite[Section 5.2]{GKRV20_bootstrap}, we define a dense subspace $\cC$ of $\cH$ as the subspace of functionals $F\in\cH$ for which there exist $N\in\N^*$ and $f\in\cC^\infty(\R\times\C^N)$ such that $F(c+\varphi)=f(c,\varphi_1,...,\varphi_N)$ holds $\d c\otimes\P_{\S^1}$-almost everywhere. Moreover, we require that $f$ and all its derivatives decay faster than any polynomial in the variable $e^{-|c|}$, and have at most polynomial growth in the other variables. We refer to $\cC$ as the space of \emph{test functions}; this space comes with a natural Fr\'echet topology, and we denote by $\cC'$ the space of continuous linear forms on $\cC$.

 Let $\alpha\in\C$. We have two commuting representations of the Heisenberg algebra $(\bA_n,\tilde{\bA}_n)_{n\in\Z}$ acting as densely defined operators on $L^2(\P_{\S^1})$, with an expression given for $n>0$ by
\begin{align*}
&\bA_n^\alpha=\frac{i}{2}\del_n;&\bA_{-n}^\alpha=\frac{i}{2}(\del_{-n}-2n\varphi_n);&&\bA_0^\alpha=\frac{i}{2}\alpha;\\
&\tilde{\bA}_n^\alpha=\frac{i}{2}\del_{-n};&\tilde{\bA}_{-n}^\alpha=\frac{i}{2}(\del_n-2n\varphi_{-n});&&\tilde{\bA}_0^\alpha=\frac{i}{2}\alpha.
\end{align*}
Here, $\del_n=\del_{\varphi_n}$ mean (complex) derivative in the direction $\varphi_n$ and $\del_{-n}=\del_{\bar{\varphi}_n}$. For $n\neq0$, we have the hermiticity relations $(\bA_n^\alpha)^*=\bA_{-n}^\alpha$ on $L^2(\P_{\S^1})$.


The \emph{Sugawara construction} consists of two commuting representations $(\bL_n^{0,\alpha},\tilde{\bL}_n^{0,\alpha})_{n\in\Z}$ of the Virasoro algebra on $L^2(\P_{\S^1})$. These operators are the following quadratic expressions in the Heisenberg algebra, for $n\in\Z\setminus\{0\}$ (recall also \eqref{eq:def_Delta}):
\begin{equation}
\begin{aligned}
&\bL_n^{0,\alpha}:=i(\alpha-(n+1)Q)\bA_n+\sum_{m\neq\{0,n\}}\bA_{n-m}\bA_m;&\bL_0^{0,\alpha}:=\Delta_\alpha+2\sum_{m=1}^\infty\bA_{-m}\bA_m;\\
&\tilde{\bL}_n^{0,\alpha}:=i(\alpha-(n+1)Q)\tilde{\bA}_n+\sum_{m\neq\{0,n\}}\tilde{\bA}_{n-m}\tilde{\bA}_m;&\bL_0^{0,\alpha}:=\Delta_\alpha+2\sum_{m=1}^\infty\tilde{\bA}_{-m}\tilde{\bA}_m.
\end{aligned}
\end{equation}
This representation satisfies the hermiticity relations $(\bL_n^{0,\alpha})^*=\bL_{-n}^{0,2Q-\bar{\alpha}}$ on $L^2(\P_{\S^1})$. Given a Young diagram $\nu=(\nu_1,...,\nu_\ell)$, we set $\bL_{-\nu}^{0,\alpha}=\bL_{-\nu_\ell}^{0,\alpha}...\bL_{-\nu_1}^{0,\alpha}$. Similar formulas and notations hold for the representation $\tilde{\bL}_n^{0,\alpha}$. The descendant state 
\[\cQ_{\alpha,\nu,\tilde{\nu}}:=\bL_{-\nu}^{0,\alpha}\tilde{\bL}_{-\tilde{\nu}}^{0,\alpha}\ind\in\cF\]
is a polynomial of level $|\nu|+|\tilde{\nu}|$.

Let $\cV_\alpha^0$ be the $(\bL_n^{0,\alpha},\tilde{\bL}_n^{0,\alpha})_{n\in\Z}$ highest-weight representation obtained by acting with the Virasoro operators on the vacuum vector, i.e. 
\[\cV_\alpha^0:=\mathrm{span}\{\bL_{-\nu}^{0,\alpha}\tilde{\bL}_{-\tilde{\nu}}^{0,\alpha}\ind|,\nu,\tilde{\nu}\in\cT\}\subset\cF,\]
where the span is algebraic (finite linear combinations). We also define $\cV_\alpha^{0,N}:=\cV_\alpha^0\cap\cF_N$. The module $\cV_\alpha^0$ has central charge $c_\mathrm{L}=1+6Q^2$ and highest-weight $\Delta_\alpha$. If $\alpha\not\in kac$, it is known that $\cV_\alpha^0$ is irreducible and Verma (hence $\cV_\alpha^0\simeq\cF$ as vector spaces) when $\alpha\not\in kac$ \cite[Lecture 8]{KacRaina_Bombay} (see also \cite[Appendix B]{BW}). On the other hand, if $\alpha\in kac^-$, $\cV_{2Q-\alpha}^0$ is Verma (hence $\cV_{2Q-\alpha}^0\simeq\cF$), and $\cV_\alpha^0$ is the irreducible quotient of the Verma by the maximal proper submodule \cite{Frenkel92_determinant} (see also \cite[Section 2.3]{BW}). The linear map
\begin{equation}\label{eq:def_Phi}
\Phi_\alpha^0:\left\lbrace\begin{aligned}
&\cV_{2Q-\alpha}^0\simeq\cF&&\to\cV_\alpha^0\\
&\bL_{-\nu}^{0,2Q-\alpha}\tilde{\bL}_{-\nu}^{0,2Q-\alpha}\ind&&\mapsto\bL_{-\nu}^{0,\alpha}\tilde{\bL}_{-\nu}^{0,\alpha}\ind
\end{aligned}\right.
\end{equation}
implements the canonical projection from the Verma to $\cV_\alpha^0$, and $\cV_\alpha^0=\mathrm{ran}\,\Phi_\alpha^0\simeq\cF/\ker\Phi_\alpha^0$.

In the sequel, we will write
\[\bS^0_\alpha:=\alpha^2\bL_{-2}^{0,\alpha}+(\bL_{-1}^{0,\alpha})^2\qquad\text{and}\qquad\tilde{\bS}^0_\alpha:=\alpha^2\tilde{\bL}_{-2}^{0,\alpha}+(\tilde{\bL}_{-1}^{0,\alpha})^2.\]

\begin{lemma}\label{lem:ff}
We have
\[\tilde{\bS}^0_\alpha\bS^0_\alpha\ind=\bS^0_\alpha\tilde{\bS}^0_\alpha\ind=\alpha^2(\alpha-\alpha_{1,2})^2(\alpha-\alpha_{2,1})^2(4|\varphi_2|^2-1).\]
\end{lemma}
\begin{proof}
This is a direct computation. We list the intermediate steps:
\begin{align*}
&\bL_{-1}^{0,\alpha}\ind=\alpha\varphi_1;\qquad(\bL_{-1}^{0,\alpha})^2\ind=\alpha^2\varphi_1^2+2\alpha\varphi_2;\qquad\bL_{-2}^{0,\alpha}\ind=2(\alpha+Q)\varphi_2-\varphi_1^2;\\
&\bS^0_\alpha\ind=2\alpha(\alpha^2+\alpha Q+1)\varphi_2=2\alpha(\alpha-\alpha_{1,2})(\alpha-\alpha_{2,1})\varphi_2;\\
&\tilde{\bL}_{-1}^{0,\alpha}\varphi_2=\alpha\bar{\varphi}_1\varphi_2;\qquad(\tilde{\bL}_{-1}^{0,\alpha})^2\varphi_2=\alpha^2\bar{\varphi}_1^2\varphi_2+\frac{1}{2}(4|\varphi_2|^2-1);\qquad\tilde{\bL}_{-2}^{0,\alpha}\varphi_2=-\bar{\varphi}_1^2\varphi_2+\frac{1}{2}(4|\varphi_2|^2-1);\\
&\tilde{\bS}^0_\alpha\bS^{0}_{\alpha}\ind=\alpha^2(\alpha-\alpha_{1,2})^2(\alpha-\alpha_{2,1})^2(4|\varphi_2|^2-1).
\end{align*}
\end{proof}

		\subsubsection{Semigroups and Poisson operator}\label{subsubsec:semigroups}
Let $X_\D$ be a Dirichlet free field in $\D$, i.e. $X_\D$ is Gaussian with covariance
\[\E[X_\D(z)X_\D(w)]=\log\left|\frac{1-z\bar{w}}{z-w}\right|=:G_\D(z,w).\]
We take this free field to be independent of $c+\varphi$. The field $X=X_\D+P\varphi$ is $\log$-correlated in $\D$:
\[\E[X(z)X(w)]=\log\frac{1}{|z-w|}=:G(z,w).\]
By Kahane's theory of Gaussian multiplicative chaos (GMC) \cite{Kahane85} and variants \cite{DuplantierSheffield11,Berestycki17,rhodes2014_gmcReview}, we can define the random measure 
\[\d M_\gamma(z)=\underset{\epsilon\to0}\lim\,\epsilon^\frac{\gamma^2}{2}e^{\gamma X_\epsilon(z)}|\d z|^2,\]
 where $X_\epsilon$ is a regularisation of the field at scale $\epsilon>0$. For concreteness, we will use the circle average regularisation \cite[Section 1.12]{Berestycki_lqggff}, but the limiting measure is unique among a large choice of regularisations \cite[Theorem 1.1]{Berestycki17}.

There are two important one-parameter semigroups of operators on $\cH$, with respective generators denoted by $\bH^0$ and $\bH$. Both generators are positive, self-adjoint, unbounded operators. They are called the \emph{free field} and \emph{Liouville semigroups} respectively. The operator $\bH$ defines a positive definite quadratic form with domain $\cD(\cQ)\subset\cH$. The continuous dual of $\cD(\cQ)$ is denoted by $\cD'(\cQ)$. Explicitly, the semigroups are defined by the probabilistic formula, for all $F\in\cH$:
\begin{equation}\label{eq:bulk_semigroups}
\begin{aligned}
&e^{-t\bH^0}F(c+\varphi)=e^{-\frac{Q^2}{2}t}\E_\varphi\left[F(c+X(e^{-t}\cdot\,))\right]\\
&e^{-t\bH}F(c+\varphi)=e^{-\frac{Q^2}{2}t}\E_\varphi\left[F(c+X(e^{-t}\cdot\,))e^{-\mu e^{\gamma c}\int_{\A_t}\frac{\d M_\gamma(z)}{|z|^{\gamma Q}}}\right],
\end{aligned}
\end{equation}
where $\A_t=\{e^{-t}<|z|<1\}$, and $\E_\varphi$ means the conditional expectation with respect to $\varphi$ (i.e. $c+\varphi$ is fixed and we integrate over $X_\D$). See \cite[Sections 4.1 \& 4.3]{GKRV20_bootstrap} for more details on the free field semigroup, and \cite[Proposition 5.1]{GKRV20_bootstrap} for the Liouville semigroup.

Informally, the \emph{Poisson operator} is a map sending free field (generalised) eigenstates to Liouville (generalised) eigenstates. The eigenstates of $\bH^0$ are easily described: for each $p\in\R$, $e^{ipc}\cF_N$ is an eigenspace with eigenvalue $2\Delta_{Q+ip}+N$. The diagonalisation of $\bH$ is much harder and is the cornerstone of the conformal bootstrap theorem \cite{GKRV20_bootstrap}. The Poisson operator is constructed using the long time asymptotics of the Liouville semigroup as follows. Let $\nu,\tilde{\nu}\in\cT$. For all $\alpha\in\C$ with $\Re(\alpha)$ sufficiently small, the limit
\[\cP_\alpha(\cQ_{\alpha,\nu,\tilde{\nu}}):=\underset{t\to\infty}\lim\,e^{t(2\Delta_\alpha+|\nu|+|\tilde{\nu}|)}e^{-t\bH}\left(e^{(\alpha-Q)c}\cQ_{\alpha,\nu,\tilde{\nu}}\right)\]
exists in a weighted space $e^{-\beta c}\cD(\cQ)$ with $\beta>Q-\Re(\alpha)$ \cite[(3.26)]{BGKRV22}. In this region, we define $\Phi_\alpha:=\cP_\alpha\circ\Phi_\alpha^0$ (recall \eqref{eq:def_Phi} for the definition of $\Phi_\alpha^0$), and we have the explicit expression
\[\Phi_\alpha(\cQ_{2Q-\alpha,\nu,\tilde{\nu}})=\Psi_{\alpha,\nu,\tilde{\nu}}\in e^{-\beta c}\cD(\cQ),\]
where the $\Psi_{\alpha,\nu,\tilde{\nu}}$ are generalised eigenstates of the Liouville Hamiltonian \cite{GKRV20_bootstrap}, with eigenvalue $2\Delta_\alpha+|\nu|+|\tilde{\nu}|$. For each $\beta>0$, the map $\alpha\mapsto\Psi_{\alpha,\nu,\tilde{\nu}}\in e^{-\beta c}\cD(\cQ)$ extends analytically to the region $\{|\Re(\alpha)-Q|<\beta\}$ and satisfies the reflection formula $\Psi_{2Q-\alpha,\nu,\tilde{\nu}}=R(2Q-\alpha)\Psi_{\alpha,\nu,\tilde{\nu}}$, with $\alpha\mapsto R(\alpha)\in\C$ a meromorphic function on $\C$ known as the \emph{reflection coefficient} \cite[Theorem 1.2]{BGKRV22}. We can view $\alpha\mapsto\Psi_{\alpha,\nu,\tilde{\nu}}$ as an analytic family in the whole $\alpha$-plane, taking values in $\cup_{\beta>0}e^{-\beta c}\cD(\cQ)$. In the sequel, when we say that $\Psi_{\alpha,\nu,\tilde\nu}$ is analytic with values in $e^{-\beta c}\cD(\cQ)$, it will be implicitly assumed that a choice of $\beta$ has been made such that the analyticity holds for this $\beta$ in the $\alpha$-region considered (this will play no special role as we will only be interested in regions with $\Re(\alpha)$ bounded below). For $\Re(\alpha)<Q$, the map $\alpha\mapsto\Phi_\alpha^0(\cQ_{2Q-\alpha,\nu,\tilde{\nu}})$ is analytic with zeros located on the Kac table \cite[Section 2.3]{BW}, so that the map $\alpha\mapsto\cP_\alpha\circ\Phi_\alpha^0(\cQ_{2Q-\alpha,\nu,\tilde{\nu}})=\cP_\alpha(\cQ_{\alpha,\nu,\tilde{\nu}})$ extends meromorphically to the whole plane, with possible poles located on the zeros of $\Phi_\alpha^0$. In the rest Section \ref{sec:bulk}, we always work with $|\nu|,|\tilde{\nu}|\leq 2$ and $\Re(\alpha)$ bounded below, so we fix once and for all a $\beta>0$ such that all the states we consider are analytic in $\alpha$ with values in~$e^{-\beta c}\cD(\cQ)$.

For $\alpha\in(-\infty,Q)$, the highest-weight state $\Psi_\alpha$ has a probabilistic representation \cite[(1.7)]{BGKRV22}\footnote{In the sequel, we will consider states with additional parameters, so we will write $\Psi_\alpha$ for $\Psi_\alpha(c,\varphi)$ to lighten the notation.}
\[\Psi_\alpha(c,\varphi)=e^{(\alpha-Q)c}\E_\varphi\left[e^{-\mu e^{\gamma c}\int_\D\frac{\d M_\gamma(z)}{|z|^{\gamma\alpha}}}\right],\]
and it admits an analytic continuation in $\alpha$ \cite[Theorem 1.2]{BGKRV22}. Following \cite[(3.2)]{BW}, for $r\in\Z_{>0}$ and non-coinciding points $w_1,...,w_r\in\D\setminus\{0\}$, we introduce the notation
\[\Psi_\alpha(w_1,...,w_r)=\underset{\epsilon\to0}\lim\,e^{(\alpha+r\gamma-Q)c}\E_\varphi\left[\epsilon^\frac{\alpha^2}{2}e^{\alpha X_\epsilon(0)}\prod_{j=1}^r\epsilon^\frac{\gamma^2}{2}e^{\gamma X_\epsilon(w_j)}e^{-\mu e^{\gamma c}\int_\D\epsilon^\frac{\gamma^2}{2}e^{\gamma X_\epsilon(z)}|\d z|^2}\right],\]
which is the primary field $\Psi_\alpha$ with additional $\gamma$-insertions at $w_1,...,w_r$. By the Cameron--Martin shift applied to the exponential prefactors, the limit can be expressed as a Laplace transform of Gaussian multiplicative chaos with singularities at $0,w_1,...,w_r$ \cite[(3.3)]{BW}. In fact, we will only focus on the case $r\leq2$. Given $\boldsymbol{s}=(s_1,...,s_r),\tilde{\boldsymbol{s}}=(\tilde{s}_1,...,\tilde{s}_r)\in\N^r$, we define, if it exists
\begin{equation}\label{eq:def_integrals}
\cI_{r,\boldsymbol{s},\boldsymbol{\tilde{s}}}(\alpha)=\int_{\D^r}\Psi_\alpha(w_1,...,w_r)\prod_{j=1}^r\frac{|\d w_j|^2}{w_j^{s_j}\bar{w}_j^{\tilde{s}_j}}.
\end{equation}

	\subsection{Expression of the singular state}\label{subsec:expression_bulk}
	
The next proposition gives a probabilistic expression for the singular state, involving singular integrals of the form \eqref{eq:def_integrals}.
\begin{proposition}\label{prop:expression_bulk}
For all $\alpha<\alpha_{1,2}$, we have the following equality in $e^{-\beta c}\cD(\cQ)$:
\[\cP_\alpha(4|\varphi_2|^2-1)=-\frac{\mu\gamma^2}{4}\cI_{1,(2),(2)}(\alpha)+\frac{\mu^2\gamma^2}{4}\cI_{2,(2,0),(0,2)}(\alpha)+\cR_\alpha,\]
where $\alpha\mapsto\cR_\alpha$ extends analytically in a complex neighbourhood of $(-\infty,0)$.
\end{proposition}

\begin{proof}
This can be proved by Gaussian integration by parts, which we view as a differential version of the Girsanov transform, following the method of \cite[Lemma 3.11]{BW}. Under the free field dynamics generated by $\bH^0$, the time evolution of the field can be represented as
\[c_t+\varphi_t(e^{i\theta})=c+P\varphi(e^{-t+i\theta})+X_\D(e^{-t+i\theta}),\]
where $X_\D$ is an independent Dirichlet GFF in $\D$. Moreover, all the modes of the field are independent, with $c_t=c+B_t$ evolving a standard Brownian motion started from $c$, and $\Re(\varphi_n(t))$, $\Im(\varphi_n(t))$ evolving as Ornstein--Uhlenbeck processes (where $\varphi_n(t)$ is defined through the expansion of $\varphi_t$ in Fourier modes: $\varphi_t(e^{i\theta})=2\Re(\sum_{n=1}^\infty\varphi_n(t)e^{ni\theta})$). See \cite[Sections 4.1 \& 4.3]{GKRV20_bootstrap} for details. 

For each $\varepsilon\in\C$, we can then consider the exponential martingale: 
\[\cE_\varepsilon(t):=e^{e^{2t}(2\Re(\varepsilon\varphi_2(t))-\frac{|\varepsilon|^2}{2}\sinh(2t))},\]
with initial value $\cE_\varepsilon(0)=e^{2\Re(\varepsilon\varphi_2)}$. This martingale was already introduced in the proof of \cite[Lemma 3.11]{BW}. Observe that 
\[\frac{\del^2}{\del\varepsilon\del\bar{\varepsilon}}_{|\varepsilon=0}e^{-\frac{|\varepsilon|^2}{4}}\cE_\varepsilon(t)=e^{4t}|\varphi_2(t)|^2-\frac{1}{2}e^{2t}\sinh(2t)-\frac{1}{4}=\frac{e^{4t}}{4}\left(4|\varphi_2(t)|^2-1\right).\]

 By Girsanov's theorem, the effect of reweighting the measure by $e^{-2\Re(\varepsilon\varphi_2)}\cE_\varepsilon(t)$ is to shift the Dirichlet free field as
\[X_\D(z)\mapsto X_\D(z)+\Re(\varepsilon(z^{-2}-\bar{z}^2)).\]
Moreover, by the Cameron--Martin formula, we have for all $F\in\cC$ (recall the standard basis introduced in~\eqref{eq:fourier}):
\[\E\left[e^{2\Re(\varepsilon\varphi_2)-\frac{|\varepsilon|^2}{4}}F(c,\varphi)\right]=\E\left[F\left(c,\varphi+\frac{1}{2}\Re(\varepsilon e_{-2})\right)\right].\]
Hence, the total shift of the field $X=X_\D+P\varphi$ is $X(z)\mapsto X(z)+\Re(\varepsilon z^{-2})$. For all $t>0$, the Feynman--Kac formula of \cite[Proposition 5.1]{GKRV20_bootstrap} and Girsanov's theorem give
\begin{align*}
&e^{t(2\Delta_\alpha+4)}\E\left[e^{-t\bH}\left((4|\varphi_2|^2-1)e^{(\alpha-Q)c}\right)F(c,\varphi)\right]\\
&\quad=4e^{(\alpha-Q)c}\frac{\del^2}{\del\varepsilon\del\bar{\varepsilon}}_{|\varepsilon=0}e^{-\frac{|\varepsilon|^2}{4}}\E\left[e^{\alpha B_t-\frac{\alpha^2}{2}t}\cE_\varepsilon(t)e^{-\mu e^{\gamma c}M_\gamma(\A_t)}F(c,\varphi)\right]\\
&\quad=4e^{(\alpha-Q)c}\frac{\del^2}{\del\varepsilon\del\bar{\varepsilon}}_{|\varepsilon=0}\E\left[e^{\alpha B_t-\frac{\alpha^2}{2}t}e^{-\mu e^{\gamma c}\int_{\A_t}e^{\frac{\gamma}{2}\Re(\varepsilon z^{-2})}\d M_\gamma(z)}F(c,\varphi+\frac{1}{2}\Re(\varepsilon e_{-2}))\right]\\
&\quad=-\frac{\mu\gamma^2}{4}e^{(\alpha+\gamma-Q)c}\int_{\A_t}\E\left[e^{\alpha B_t-\frac{\alpha^2}{2}t}e^{\gamma X(w)-\frac{\gamma^2}{2}\E[X(w)^2]}e^{-\mu e^{\gamma c}M_\gamma(\A_t)}F(c,\varphi)\right]\frac{|\d w|^2}{|w|^4}\\
&\qquad+\frac{\mu^2\gamma^2}{4}e^{(\alpha+2\gamma-Q)c}\int_{\A_t}\int_{\A_t}\E\left[e^{\alpha B_t-\frac{\alpha^2}{2}t}e^{\gamma X(w_1)-\frac{\gamma^2}{2}\E[X(w_1)^2]}e^{\gamma X(w_2)-\frac{\gamma^2}{2}\E[X(w_2)]^2}e^{-\mu e^{\gamma c}M_\gamma(\A_t)}F(c,\varphi)\right]\frac{|\d w_1|^2}{\bar{w}_1^2}\frac{|\d w_2|^2}{w_2^2}\\
&\qquad-\frac{\mu\gamma}{4}e^{(\alpha+\gamma-Q)c}\int_{\A_t}\E\left[e^{\alpha B_t-\frac{\alpha^2}{2}t}e^{\gamma X(w)-\frac{\gamma^2}{2}\E[X(w)^2]}e^{-\mu e^{\gamma c}M_\gamma(\A_t)}\del_2F(c,\varphi)\right]\frac{|\d w|^2}{w^2}\\
&\qquad-\frac{\mu\gamma}{4}e^{(\alpha+\gamma-Q)c}\int_{\A_t}\E\left[e^{\alpha B_t-\frac{\alpha^2}{2}t}e^{\gamma X(w)-\frac{\gamma^2}{2}\E[X(w)^2]}e^{-\mu e^{\gamma c}M_\gamma(\A_t)}\del_{-2}F(c,\varphi)\right]\frac{|\d w|^2}{\bar{w}^2}\\
&\qquad+\frac{1}{4}e^{(\alpha-Q)c}\E\left[e^{\alpha B_t-\frac{\alpha^2}{2}t}e^{-\mu e^{\gamma c}M_\gamma(\A_t)}\del_2\del_{-2}F(c,\varphi)\right]\\
&\underset{t\to\infty}\to-\frac{\mu\gamma^2}{4}\E[\cI_{1,(2),(2)}(\alpha)F]+\frac{\mu^2\gamma^2}{4}\E[\cI_{2,(2,0),(0,2)}(\alpha)F]+\E[\cR_\alpha F].
\end{align*}
For $\alpha$ very negative, the above convergence is readily seen using the Cameron--Martin shift and the representations of the states in \cite[(3.3)]{BW}.  By \cite[Proposition 3.2]{BW}, $\cR_\alpha$ is analytic for $\alpha$ in a complex neighbourhood of $(-\infty,0)$, which concludes the proof by density of $\cC$ in $L^2(\P_{\S^1})$. 
\end{proof}

	\subsection{First pole of $\cP_\alpha$ and the $(1,2)$-HEM}\label{subsec:first_pole_bulk}

The $(1,2)$-HEM will be a simple consequence of the previous proposition. We just need to evaluate the residue of $\cP_\alpha(4|\varphi_2|^2-1)$ at $\alpha_{1,2}$. 
	
\begin{proposition}\label{prop:first_pole}
The following holds in $e^{-\beta c}\cD(\cQ)$:
\[\underset{\alpha=\alpha_{1,2}}{\mathrm{Res}}\,\cP_\alpha(4|\varphi_2|^2-1)=\underset{\alpha\to\alpha_{1,2}}\lim\,(\alpha-\alpha_{1,2})\cP_\alpha(4|\varphi_2|^2-1)=\pi\mu\frac{\gamma}{2}\Psi_{\alpha_{-1,2}}.\]
\end{proposition}

\begin{proof}
By \cite[Proposition 3.2]{BW}, $\cI_{2,(2,0),(0,2)}(\alpha)$ is regular at $\alpha_{1,2}$, so the only term of Proposition \ref{prop:expression_bulk} with a possible pole at $\alpha_{1,2}$ is $\cI_{1,(2),(2)}(\alpha)$. 

Let $F\in\cC$. By Proposition \ref{prop:fusion} applied to $r=1$, we have in $e^{-\beta c}\cD(\cQ)$:
\[|\E[(\Psi_\alpha(w)-|w|^{-\gamma\alpha}\Psi_{\alpha+\gamma})F]|=O_{w\to0}(|w|^{-\gamma\alpha+\xi})\]
for some $\xi>0$, uniformly in a neighbourhood of $\alpha_{1,2}$. Using this estimate, we have
\begin{equation}\label{eq:residue_I_one}
\begin{aligned}
\E[\cI_{1,(2),(2)}(\alpha)F]
&=\E[\Psi_{\alpha+\gamma}F]\int_\D|w|^{-\gamma\alpha}\frac{|\d w|^2}{|w|^4}+\int_\D\E[(\Psi_\alpha(w)-|w|^{-\gamma\alpha}\Psi_{\alpha+\gamma})F]\frac{|\d w|^2}{|w|^4}\\
&=-\frac{2\pi}{\gamma(\alpha-\alpha_{1,2})}\E[\Psi_{\alpha+\gamma}F]+O_{\alpha\to\alpha_{1,2}}(1),
\end{aligned}
\end{equation}
with the $O(1)$ coming from the integrability of $|w|^{-4-\gamma\alpha+\xi}$ in $\D$, for $\alpha$ in a neighbourhood of $\alpha_{1,2}$.

This shows that $\E[\cI_{1,(2),(2)}(\alpha)F]$ is meromorphic in a neighbourhood of $\alpha_{1,2}$ with a simple pole there, so that $\underset{\alpha=\alpha_{1,2}}{\mathrm{Res}}\,\cI_{1,(2),(2)}(\alpha)=-\frac{2\pi}{\gamma}\Psi_{\alpha_{-1,2}}$ in $\cC'$ (recall $\alpha_{-1,2}=\alpha+\gamma$). Combining with Proposition \ref{prop:expression_bulk}, we get $\underset{\alpha=\alpha_{1,2}}{\mathrm{Res}}\,\cP_\alpha(4|\varphi_2|^2-1)=\pi\mu\frac{\gamma}{2}\Psi_{\alpha_{-1,2}}$ in $\cC'$. 

 Moreover, we have explained in Section \ref{subsec:bulk_setup} that $\cP_\alpha(4|\varphi_2|^2-1)$ is meromorphic with values in $e^{-\beta c}\cD(\cQ)$, so the residue exists in $e^{-\beta c}\cD(\cQ)$. Now, the function $(\alpha-\alpha_{1,2})\cI_{1,(2),(2)}(\alpha)$ is analytic in a neighbourhood of $\alpha_{1,2}$ so that Cauchy's formula gives $(\alpha-\alpha_{1,2})\cI_{1,(2),(2)}(\alpha)=\frac{1}{2i\pi}\oint\cI_{1,(2),(2)}(\alpha')\frac{\alpha'-\alpha_{1,2}}{\alpha'-\alpha}\d\alpha'$ in $\cC'$ in a neighbourhood of $\alpha=\alpha_{1,2}$, where the contour is a small loop surrounding $\alpha_{1,2}$. Now, the right-hand-side is an integral of elements in $e^{-\beta c}\cD(\cQ)$ and avoids the singularity at $\alpha_{1,2}$, so it converges to $\frac{1}{2i\pi}\oint\cI_{1,(2),(2)}(\alpha)\d\alpha=\underset{\alpha=\alpha_{1,2}}{\mathrm{Res}}\cI_{1,(2),(2)}(\alpha)$ in~$e^{-\beta c}\cD(\cQ)$.
\end{proof}

From here, we can conclude the proof of the $(1,2)$-HEM.
\begin{proof}[Proof of \eqref{eq:bulk_(1,2)}]
By Lemma \ref{lem:ff}, $\alpha\mapsto\bS^0_\alpha\tilde{\bS}^0_\alpha\ind$ has a zero of order 2 at $\alpha_{1,2}$. Hence, by the previous proposition, $\cP_\alpha(\bS^0_\alpha\tilde{\bS}^0_\alpha\ind)$ has a zero of order 1 at $\alpha_{1,2}$. The derivative at $\alpha_{1,2}$ is then given by
\begin{align*}
\frac{\del}{\del\alpha}_{|\alpha=\alpha_{1,2}}\cP_\alpha(\bS^0_\alpha\tilde{\bS}^0_\alpha\ind)
=\underset{\alpha\to\alpha_{1,2}}\lim\,\frac{1}{\alpha-\alpha_{1,2}}\cP_\alpha(\bS^0_\alpha\tilde{\bS}^0_\alpha\ind)
&=\underset{\alpha\to\alpha_{1,2}}\lim\,(\alpha-\alpha_{1,2})\alpha^2(\alpha-\alpha_{2,1})^2\cP_\alpha(4|\varphi_2|^2-1)\\
&=\alpha_{1,2}^2(\alpha_{1,2}-\alpha_{2,1})^2\underset{\alpha=\alpha_{1,2}}{\mathrm{Res}}\,\cP_\alpha(4|\varphi_2|^2-1)\\
&=\pi\mu\frac{8}{\gamma^3}\left(1-\frac{\gamma^2}{4}\right)^2\Psi_{\alpha_{-1,2}}.
\end{align*}

On the other hand, \cite[Theorem 4.5]{BGKRV22} states that the Liouville descendant states are obtained from the primary state $\Psi_\alpha$ by acting with the Virasoro generators of the Liouville theory: in other words, the Poisson operator intertwines the Virasoro algebras of the free field theory and the Liouville theory in the sense that $\cP_\alpha\circ\bL_n^{0,\alpha}=\bL_n^\alpha\circ\cP_\alpha$, with $\bL_n^\alpha$ the operator $\bL_n$ acting in the module $\cW_\alpha$ defined in \cite[Section 4.4]{BGKRV22}. The same also holds with the representation $(\tilde{\bL}_n)_{n\in\Z}$. Combining with $\tilde{\bS}_{\alpha_{1,2}}\Psi_{\alpha_{1,2}}=0$ \cite[Theorem 1.2]{BW}, we have
\[\frac{\del}{\del\alpha}_{|\alpha=\alpha_{1,2}}\cP_\alpha(\bS^0_\alpha\tilde{\bS}^0_\alpha\ind)=
\underset{\alpha\to\alpha_{1,2}}\lim\,\frac{1}{\alpha-\alpha_{1,2}}\bS_\alpha\tilde{\bS}_\alpha\Psi_\alpha
=\underset{\alpha\to\alpha_{1,2}}\lim\,\bS_\alpha\tilde{\bS}_\alpha\left(\frac{\Psi_\alpha-\Psi_{\alpha_{1,2}}}{\alpha-\alpha_{1,2}}\right)=\bS_{\alpha_{1,2}}\tilde{\bS}_{\alpha_{1,2}}\Psi_{\alpha_{1,2}}'.\]
\end{proof}

	\subsection{Second pole of $\cP_\alpha$ and the $(2,1)$-HEM}\label{subsec:second_pole_bulk}
In this section, we compute the residue of $\cP_\alpha(4|\varphi_2|^2-1)$ at $\alpha=\alpha_{2,1}$, which will prove \eqref{eq:bulk_(2,1)} and end the proof of Theorem \ref{thm:bulk}. It is the content of the following proposition.

\begin{proposition}\label{prop:plan_21}
We have in $e^{-\beta c}\cD(\cQ)$ for $\alpha\in(-\infty,\alpha_{1,2})$
\begin{equation}\label{eq:mero}
(\alpha-\alpha_{1,2})^2\cP_\alpha(4|\varphi_2|^2-1)=\frac{\mu^2\gamma^4}{16}\cI_{2,(1,1),(1,1)}(\alpha)+\cR_\alpha,
\end{equation}
where $\cR_\alpha$ is analytic up to a neighbourhood of $\alpha_{2,1}$. Moreover, $\cI_{2,(1,1),(1,1)}$ is meromorphic in a neighbourhood of $(-\infty,\alpha_{2,1})$, with
\begin{equation}\label{eq:residue_I_two}
\underset{\alpha=\alpha_{2,1}}{\mathrm{Res}}\,\cI_{2,(1,1),(1,1)}(\alpha)=\left\lbrace\begin{aligned}
&-\frac{2}{\gamma}\left(\pi\frac{\Gamma(\frac{\gamma^2}{4})}{\Gamma(1-\frac{\gamma^2}{4})}\right)^2\frac{\Gamma(1-\frac{\gamma^2}{2})}{\Gamma(\frac{\gamma^2}{2})}\Psi_{\alpha_{-2,1}},&\text{if }\gamma<\sqrt{2};\\
&0,&\text{if }\gamma>\sqrt{2}.
\end{aligned}\right.
\end{equation}
In particular, $\cI_{2,(1,1),(1,1)}$ is regular at $\alpha_{2,1}$ if $\gamma>\sqrt{2}$.
\end{proposition}

Assuming this proposition, we can easily conclude the proof of the $(2,1)$-HEM.
\begin{proof}[Proof of \eqref{eq:bulk_(2,1)}]
As in the previous proof, the intertwining relation of \cite[Theorem 4.5]{BGKRV22} and Proposition~\ref{prop:plan_21} give
\begin{align*}
\bS_{\alpha_{2,1}}\tilde{\bS}_{\alpha_{2,1}}\Psi_{\alpha_{2,1}}'=\frac{\del}{\del\alpha}_{|\alpha=\alpha_{2,1}}\cP_\alpha(\bS_\alpha^0\tilde{\bS}_\alpha^0\ind)
&=\underset{\alpha\to\alpha_{2,1}}\lim\,\frac{1}{\alpha-\alpha_{2,1}}\cP_\alpha(\bS_\alpha^0\tilde{\bS}_\alpha^0\ind)\\
&=\underset{\alpha\to\alpha_{2,1}}\lim\,(\alpha-\alpha_{2,1})\alpha_{2,1}^2(\alpha_{2,1}-\alpha_{1,2})^2\cP_\alpha(4|\varphi_2|^2-1)\\
&=-\frac{\gamma^5}{32}\left(\pi\mu\frac{\Gamma(\frac{\gamma^2}{4})}{\Gamma(1-\frac{\gamma^2}{4})}\right)^2\frac{\Gamma(1-\frac{\gamma^2}{2})}{\Gamma(\frac{\gamma^2}{2})}\Psi_{\alpha_{-1,2}}
\end{align*}
\end{proof}

The remainder of Section \ref{subsec:second_pole_bulk} is devoted to the proof of Proposition \ref{prop:plan_21}. First, we show that the residue of $\cI_{2,(1,1),(1,1)}$ at $\alpha_{2,1}$ factorises into the primary field $\Psi_{\alpha_{-1,2}}$ and a Dotsenko--Fateev integral. This is done thanks to the fusion estimate of Proposition \ref{prop:fusion}. The residue of the Dotsenko--Fateev integral can be computed explicitly (Lemma \ref{lem:dotsenko_fateev}), and we get \eqref{eq:residue_I_two}. Finally, Proposition \ref{prop:mero} establishes the expression of \eqref{eq:mero} for the meromorphic continuation, using a method inspired by \cite{BW}. Our method for the computation of the scalar multiple of $\Psi_{\alpha_{-2,1}}$ differs (and is somewhat more natural) from the one proposed in \cite{Zamolodchikov03_HEM}: we do purely free field computations (requiring only the exact value of the Dotsenko--Fateev integral), while \cite[Section 5]{Zamolodchikov03_HEM} makes a detour through Liouville (using the DOZZ formula and non-trivial properties of the $\Upsilon$-function and Virasoro representation theory) before  coming back to the free field. 

		\subsubsection{Residues of Dotsenko--Fateev integrals}\label{subsubsec:dotsenko_fateev}
In this section, we always assume $\gamma\in(0,\sqrt{2})$. We define the integrals
\begin{equation}\label{eq:def_J}
\begin{aligned}
&J_1(\alpha):=\int_{\D^2}|w_1|^{-\gamma\alpha}|w_2|^{-\gamma\alpha}|w_1-w_2|^{-\gamma^2}\frac{|\d w_1|^2}{|w_1|^2}\frac{|\d w_2|^2}{|w_2|^2};\\
&J_2(\alpha):=\int_{\D^2}|w_1|^{-\gamma\alpha}|w_2|^{-\gamma\alpha}|w_1-w_2|^{-\gamma^2}\frac{\bar{w}_2-\bar{w}_1}{w_2-w_1}\frac{|\d w_1|^2}{w_1\bar{w}_1^2}\frac{|\d w_2|^2}{\bar{w}_2};\\
&J_3(\alpha):=\int_{\D^2}|w_1|^{-\gamma\alpha}|w_2|^{-\gamma\alpha}|w_1-w_2|^{-\gamma^2}\frac{\bar{w}_2-\bar{w}_1}{w_2-w_1}\frac{|\d w_1|^2}{w_1\bar{w}_1}\frac{|\d w_2|^2}{\bar{w}_2^2}.
\end{aligned}
\end{equation}
The integrals $J_1$ and $J_3$ are absolutely convergent for $\Re(\alpha)<\alpha_{2,1}=-\frac{\gamma}{2}$. The integral $J_2$ is absolutely convergent for $\Re(\alpha)<-\frac{1}{\gamma}$.
		
\begin{lemma}\label{lem:dotsenko_fateev}
The integrals $J_1(\alpha)$, $J_2(\alpha)$, $J_3(\alpha)$ have a meromorphic continuation in a neighbourhood of $\alpha_{2,1}$ (still denoted by the same letter). 

The function $J_1$ has a simple pole at $\alpha_{2,1}$, and 
\begin{align*}
&\underset{\alpha=\alpha_{2,1}}{\mathrm{Res}}\,J_1(\alpha)=-\frac{2}{\gamma}\left(\pi\frac{\Gamma(\frac{\gamma^2}{4})}{\Gamma(1-\frac{\gamma^2}{4})}\right)^2\frac{\Gamma(1-\frac{\gamma^2}{2})}{\Gamma(\frac{\gamma^2}{2})}.
\end{align*}

The function $J_2+J_3$ is regular at $\alpha_{2,1}$.
\end{lemma}

\begin{proof}
\emph{Residue of $J_1$.}

For $D\in\{\D,\C\}$, we set
\[J_D^\beta(\alpha):=\int_{D^2}|w_1|^{-\gamma\alpha}|w_2|^{-\gamma\alpha}|w_1-w_2|^{-\gamma^2}|1-w_1|^{-\gamma\beta}|1-w_2|^{-\gamma\beta}\frac{|\d w_1|^2}{|w_1|^2}\frac{|\d w_2|^2}{|w_2|^2},\]
which is well-defined and analytic for $\Re(\alpha)\in(\alpha_{2,1},0)$ and $\Re(\beta)\in(\alpha_{2,1}-\alpha,\alpha_{2,1}+\frac{2}{\gamma})$ if $D=\C$. If $D=\D$, we have the same bound on $\alpha$, but $J_\D$ is actually analytic all the way to $\beta$ in a neighbourhood of 0 (in fact, we want to compute $\mathrm{Res}_{\alpha=\alpha_{2,1}}J_\D^{\beta=0}(\alpha)$).  
In the region of absolute convergence, $J_\C^\beta(\alpha)$ is the Dotsenko--Fateev integral of Appendix \ref{app:selberg}, whose meromorphic continuation is given by \eqref{eq:neretin}:

\begin{equation}\label{eq:dotsenko_fateev}
\begin{aligned}
J_\C^\beta(\alpha)
&=S_{2,2}\left(-\frac{\gamma\alpha}{2},1-\frac{\gamma\beta}{2},-\frac{\gamma^2}{4}\right)^2\frac{\sin(\pi\frac{\gamma\alpha}{2})\sin(\pi\frac{\gamma\beta}{2})\sin(\pi\frac{\gamma}{2}(\alpha+\frac{\gamma}{2}))\sin(\pi\frac{\gamma}{2}(\beta+\frac{\gamma}{2}))\sin(\pi\frac{\gamma^2}{2})}{\sin(\pi\frac{\gamma}{2}(\alpha+\beta+\frac{\gamma}{2}))\sin(\pi\frac{\gamma}{2}(\alpha+\beta+\gamma))\sin(\pi\frac{\gamma^2}{4})}\\
&=-\frac{2}{\gamma}\pi\frac{\beta}{(\alpha+\frac{\gamma}{2})(\alpha+\beta+\frac{\gamma}{2})}\left(\frac{\Gamma(\frac{\gamma^2}{4})\Gamma(1-\frac{\gamma^2}{2})}{\Gamma(1-\frac{\gamma^2}{4})}\right)^2\sin(\pi\frac{\gamma^2}{2})(1+o(1)),
\end{aligned}
\end{equation}
where $o(1)$ is as $(\alpha,\beta)\to(\alpha_{2,1},0)$.

On the other hand, from the change of variables $w_j\mapsto\frac{1}{w_j}$, $j=1,2$, we have
\begin{align*}
J_\C^\beta(\alpha)=J_\D^\beta(\alpha)+J_\D^\beta(-\alpha-\gamma-\beta)+2\int_\D\int_{\C\setminus\bar{\D}}|w_1|^{-\gamma\alpha}|w_2|^{-\gamma\alpha}|w_1-w_2|^{-\gamma^2}|1-w_1|^{-\gamma\beta}|1-w_2|^{-\gamma\beta}\frac{|\d w_1|^2}{|w_1|^2}\frac{|\d w_2|^2}{|w_2|^2}.
\end{align*}
The last integral is regular at $\alpha=\alpha_{2,1}$ for all $\beta$ in a neighbourhood of 0 (indeed, the first pole is at $\alpha=0$ uniformly in $\beta$, and $\gamma<\sqrt{2}$). Moreover, since $J_\D^\beta$ has a pole at $\alpha_{2,1}$, the second term of the previous displayed equation has a pole at $\alpha_{2,1}-\beta$. Now, note that $J_\D^\beta$ converges uniformly on compacts away from $\alpha_{2,1}$ as $\beta\to0$, so $J_\D^\beta$ does not have a pole at $\alpha_{2,1}-\beta$ for small enough $\beta>0$. It follows that $\mathrm{Res}_{\alpha=\alpha_{2,1}}J_\D^\beta(\alpha)=\mathrm{Res}_{\alpha=\alpha_{2,1}}J_\C^\beta(\alpha)$ for sufficiently small $\Re(\beta)>0$. Taking $\beta\to0$ gives  
\begin{align*}
\mathrm{Res}_{\alpha=\alpha_{2,1}}\,J_\D(\alpha)=-\frac{2\pi}{\gamma}\left(\frac{\Gamma(\frac{\gamma^2}{4})\Gamma(1-\frac{\gamma^2}{2})}{\Gamma(1-\frac{\gamma^2}{4})}\right)^2\sin(\pi\frac{\gamma^2}{2})=-\pi^2\frac{2}{\gamma}\left(\frac{\Gamma(\frac{\gamma^2}{4})}{\Gamma(1-\frac{\gamma^2}{4})}\right)^2\frac{\Gamma(1-\frac{\gamma^2}{2})}{\Gamma(\frac{\gamma^2}{2})}.
\end{align*}

\emph{Regularity of $J_2+J_3$.}

Fix $\alpha$ with $\Re(\alpha)<\alpha_{2,1}$.

 For $J_2$, we do an integration by parts with respect to $\partial_{\bar{w}_2}((w_1-w_2)^{-\frac{\gamma^2}{2}-1}(\bar{w}_1-\bar{w}_2)^{-\frac{\gamma^2}{2}+2})$
\begin{align*}
	J_2
	=&-\frac{-\frac{\gamma\alpha}{2}-1}{\frac{\gamma^2}{2}-2}\int_{\D^2}|w_1|^{-\gamma\alpha}|w_2|^{-\gamma\alpha}|w_1-w_2|^{-\gamma^2}\frac{(\bar w_1-\bar w_2)^2}{w_1-w_2}\frac{|\d w_1|^2}{w_1\bar w_1^2}\frac{|\d w_2|^2}{\bar w_2^2}\\
	&\quad+\frac{i}{2}\frac{1}{\frac{\gamma^2}{2}-2}\int_{\S^1}\int_{\D}|w_2|^{-\gamma\alpha}|w_1-w_2|^{-\gamma^2}\frac{(\bar w_1-\bar w_2)^2}{w_1-w_2}\frac{|\d w_1|^2}{\bar w_1}\frac{\d w_2}{\bar w_2}.
\end{align*}

For $J_3$, we do integration by parts with respect to  $\partial_{\bar{w}_1}((w_1-w_2)^{-\frac{\gamma^2}{2}-1}(\bar{w}_1-\bar{w}_2)^{-\frac{\gamma^2}{2}+2})$
\begin{align*}
	J_3
	&=-\frac{-\frac{\gamma\alpha}{2}-1}{2-\frac{\gamma^2}{2}}\int_{\D^2}|w_1|^{-\gamma\alpha}|w_2|^{-\gamma\alpha}|w_1-w_2|^{-\gamma^2}\frac{(\bar w_1-\bar w_2)^2}{w_1-w_2}\frac{|\d w_1|^2}{w_1\bar w_1^2}\frac{|\d w_2|^2}{\bar w_2^2}\\
	&\quad+\frac{i}{2}\frac{1}{2-\frac{\gamma^2}{2}}\int_{\D}\int_{\S^1}|w_2|^{-\gamma\alpha}|w_1-w_2|^{-\gamma^2}\frac{(\bar w_1-\bar w_2)^2}{w_1-w_2}\d w_1\frac{|\d w_2|^2}{\bar w_2^2}.
\end{align*}

Then we see the bulk terms of $J_2$ and $J_3$ cancel, and the boundary terms are absolutely convergent at $\alpha_{2,1}$.
\end{proof}


		\subsubsection{Fusion estimates}\label{subsubsec:fusion}
We now turn to the fusion estimates, which are key to the factorisation of the residues as a product of a primary field and a Dotsenko--Fateev integral. These type of estimates are not new, see e.g. \cite[Section 5.1]{KRV19_local}.

The next proposition gives some fusion estimates when we send a number $r\in\N^*$ of $\gamma$-insertions to zero. We will only be needing the result for $r\leq3$, but since the proof is similar in all cases, we prefer to state it for arbitrary $r$. Recall that $\cC'$ is the continuous dual of $\cC$.

\begin{proposition}\label{prop:fusion}
Let $r\in\N^*$, and $\alpha<Q$. The following estimates hold in $\cC'$.
\begin{enumerate}[label={\arabic*.}]
\item \label{item:geq} Suppose $\alpha+r\gamma\geq Q$. Uniformly over the set of non-coinciding points $\boldw=(w_1,...,w_r)\in(\D\setminus\{0\})^r$:
\[\Psi_\alpha(\boldw)=\prod_{j=1}^r|w_j|^{-\gamma\alpha}\prod_{1\leq k<l\leq r}|w_k-w_l|^{-\gamma^2}O(\max_j|w_j|^{\frac{1}{2}(\alpha+r\gamma-Q)^2}).\]
\item \label{item:less} Let $\Delta^r:=\{\boldw\in(\D\setminus\{0\})^r|\,\forall j=1,...,r-1,\,|w_j|\geq e|w_{j+1}|\}$. There exists $\xi>0$ such that, uniformly in the region $\Delta^r$:
\[\Psi_\alpha(\boldw)=\prod_{j=1}^r|w_j|^{-\gamma\alpha}\prod_{1\leq k<l\leq r}|w_k-w_l|^{-\gamma^2}\left(\Psi_{\alpha+r\gamma}+O(\max_j|w_j|^\xi)\right).\]

\end{enumerate}
\end{proposition}

\begin{remark}
In the first online version of this work, the second item was stated without the condition $\boldw\in\Delta^r$. This is necessary as soon as $2\gamma\geq Q$ because in this case the estimates break down in the region $|w_1-w_2|\ll\min\{|w_1|,|w_2|\}$. We thank a careful referee for spotting the mistake.
\end{remark}

\begin{proof}
By permutation symmetry, we can assume $|w_1|\geq|w_2|...\geq|w_r|$ (which is already implied in the second item of the proposition), and we set $t_j:=-\log|w_j|$ for each $j=1,...,r$. Let $F\in\cC$. By the Girsanov transform, we can write
\begin{align*}
\E[\Psi_\alpha(\boldw)F]
&=e^{(\alpha+r\gamma-Q)c}\prod_{j=1}^r|w_j|^{-\gamma\alpha}\prod_{1\leq k<l\leq r}|w_k-w_l|^{-\gamma^2}\\
&\qquad\times\E\left[e^{-\mu e^{\gamma c}\int_\D\prod_{j=1}^r|z-w_j|^{-\gamma^2}\frac{\d M_\gamma(z)}{|z|^{\gamma\alpha}}}F\left(c+\varphi+\gamma\sum_{j=1}^rG(w_j,\cdot\,)\right)\right]
\end{align*}
Our goal will be to bound the expectation in the last formula. We will first assume that $F=1$ and then explain at the end what are the changes.

Recall the decomposition already used in the proof of Proposition \ref{prop:expression_bulk}: we have $X(e^{-t+i\theta})=B_t+\varphi_t(e^{i\theta})$, where the circle average process $B_t=\int_0^{2\pi}X(e^{-t+i\theta})\frac{\d\theta}{2\pi}$ is distributed as a standard Brownian motion independent of $(\varphi_t)_{t\geq0}$. We will write $Y(e^{-t+i\theta}):=\varphi_t(e^{i\theta})$. Away from 0, $Y$ differs from $X$ by a continuous function, so its GMC is almost surely well-defined in each annulus $\A_t=\{e^{-t}<|z|<1\}$, and we define $Z_t:=\lim_{\epsilon\to0}\int_{\A_t}\epsilon^\frac{\gamma^2}{2}e^{\gamma Y_\epsilon(z)}|\d z|^2$. Note that $t\mapsto Z_t$ is a.s. continuous and increasing on $\R_+$, so $\d Z_t$ is a (random) measure on $\R_+$. Then, the GMC measure of $X$ satisfies $M_\gamma(\A_t)=\int_0^te^{\gamma B_s}\d Z_s$.

\emph{Proof of item \ref{item:geq}}

We need to analyse the expectation in the previous displayed equation, which is the Laplace transform of the GMC of the shifted field $X(z)-\gamma\sum_{j=1}^r\log|z-w_j|$. Its circle average process is $B_t+\gamma\sum_{j=1}^rt\wedge t_j$, where $t_j=-\log|w_j|$. We are going to apply the Cameron--Martin theorem to this drifted Brownian motion in the interval $[0,t_1-2]$: after this shift, we get a standard Brownian motion $B_t$ in $[0,t_1-2]$. In the interval $[t_1-2,t_1-1]$, we get a drifted Brownian motion started from $B_{t_1-2}$. By the Markov property, this process can be written $B_{t_1-2}+\tilde{B}_{t-t_1+2}+(\alpha+r\gamma-Q)(t-t_1+2)$, where $(\tilde{B}_t)_{t\geq0}$ is an independent standard Brownian motion. Moreover, in the annulus $\A_{t_1-1}$, the function $|\log|z|-\log|z-w_j||$ is bounded for all $j$, uniformly in $t_1$. Then, a simple application of Kahane's convexity inequality (see e.g. \cite[Theorem 2.1]{rhodes2014_gmcReview}) shows that the next estimate holds for $Y$ or its shift indifferently, up to a global multiplicative constant. Combining this with the non-negativity of the GMC measure, we have for some constant $C>0$:

 \begin{equation}\label{eq:BIG}
\begin{aligned}
\E\left[\exp\left(-\mu e^{\gamma c}\int_\D\prod_{j=1}^r|z-w_j|^{-\gamma^2}\frac{\d M_\gamma(z)}{|z|^{\gamma\alpha}}\right)\right]
&\leq \E\left[\exp\left(-\mu e^{\gamma c}\int_{\A_{t_1-1}}\prod_{j=1}^r|z-w_j|^{-\gamma^2}\frac{\d M_\gamma(z)}{|z|^{\gamma\alpha}}\right)\right]\\
&\leq C\E\left[\exp\left(-\mu e^{\gamma c}\int_{t_1-2}^{t_1-1}e^{\gamma(B_t+\alpha+r\gamma-Q)t}\d Z_t\right)\right]
\end{aligned}
\end{equation}
Applying Cameron--Martin's theorem to the drifted Brownian motion inside the last line, we get that the last expectation in the previous displayed equation equals
\[\E\left[e^{(\alpha+r\gamma-Q)B_{t_1-2}-\frac{1}{2}(\alpha+r\gamma-Q)^2(t_1-2)}\exp\left(-\mu e^{\gamma(c+B_{t_1-2})}\int_0^1e^{\gamma(\tilde{B}_t+(\alpha+r\gamma-Q)t)}\d Z_{t+t_1-2}\right)\right]\]
Now, we disintegrate with respect to the law of $B_{t_1-2}$ and use that the Gaussian density is bounded by $(2\pi(\log\frac{1}{|w_1|}-2))^{-1/2}$ to get that the previous displayed equation is bounded by a constant times 
\[\frac{|w_1|^{\frac{1}{2}(\alpha+r\gamma-Q)^2}}{|\log|w_1||^{1/2}}\E\left[\int_\R e^{(\alpha+r\gamma-Q)x}\exp\left(-\mu e^{\gamma(c+x)}\int_0^1e^{\gamma(\tilde{B}_t+(\alpha+r\gamma-Q)t)}\d Z_{t+t_1-2}\right)\d x\right].\]
Finally, by the stationarity of the process $(Z_t)_{t\geq0}$, we can replace $Z_{t+t_1-2}$ in the last displayed equation by $Z_t$. Moreover, we can evaluate the integral over $x\in\R$ using the formula $y^{-s}=\frac{1}{\Gamma(s)}\int_\R e^{-ty}t^s\frac{\d t}{t}$ for all $y,s>0$. Plugging these estimates into \eqref{eq:BIG}, we have just proved that  
\begin{align*}
&\E\left[\exp\left(-\mu e^{\gamma c}\int_\D\prod_{j=1}^r|z-w_j|^{-\gamma^2}\frac{\d M_\gamma(z)}{|z|^{\gamma\alpha}}\right)\right]\\
&\quad\leq C\frac{|w_1|^{\frac{1}{2}(\alpha+r\gamma-Q)^2}}{|\log|w_1||^{1/2}}\E\left[\left(-\mu e^{\gamma c}\int_0^1e^{\gamma(\tilde{B}_t+(\alpha+r\gamma-Q)t)})\d Z_t\right)^{-\frac{1}{\gamma}(\alpha+r\gamma-Q)}\right]=O(|w_1|^{\frac{1}{2}(\alpha+r\gamma-Q)^2}).
\end{align*}

 Finally, for an arbitrary $F\in\cC$, we can assume that $F>0$ up to decomposing into positive and negative parts. Then, we can fix $N>0$ such that $F$ depends only on the first $N$ Fourier modes of the boundary field. Conditionally on these first modes, $F$ is a constant, so we can apply the previous estimates (including Kahane's convexity inequality). We can integrate out these last modes at the very end and the estimate remains unchanged.

\emph{Proof of item \ref{item:less}}

For non-coinciding points $w_1,...,w_r\in\D^r$, we introduce the notation
\[I_D(\boldw):=\int_D\prod_{j=1}^r|z-w_j|^{-\gamma^2}\frac{\d M_\gamma(z)}{|z|^{\gamma\alpha}},\]
for every open set $D\subset\D$. This integral admits positive moments of order $p>0$ for all $p<\min\{\frac{4}{\gamma^2},\frac{2}{\gamma}(Q-\alpha-r\gamma)\}$ \cite[(2.14)]{KRV_DOZZ}. We will assume $p<1$ so that the condition is just $p<\frac{2}{\gamma}(Q-\alpha-r\gamma)$, since $\frac{4}{\gamma^2}>1$. By an elementary computation of the covariance, we have $X|_{e^{-t}\D}(e^{-t}\,\cdot)\laweq X+\sqrt{t}\delta$, with $\delta$ an independent standard Gaussian. From this observation and the change of variable $z\mapsto e^tz$ in the definition of $I_{e^{-t}\D}$, we get the scaling relation $\E[I_{e^{-t}\D}(\boldw)^p]=e^{-t\xi(p)}\E[I_\D(e^t\boldw)^p]$, with $\xi(p)=\gamma(Q-\alpha-r\gamma)p-\frac{\gamma^2}{2}p^2$. The number $\xi(p)$ is called a multifractal exponent (see \cite[Theorem 2.14]{rhodes2014_gmcReview} for more on this topic). We note that $\sup_{\boldw\in\Delta^r}\E[I_\D(\boldw)^p]<\infty$ since the $\gamma$-insertions are far away from each other on $\Delta^r$, compared to their distance to 0.

Our goal is to control $\E[e^{-\mu I_\D(\boldw)}-e^{-\mu I_\D(0)}]$. First, we show that we can remove the contribution of the disc of radius $|w_1|^{1-\eta}$ around 0, for any $\eta\in(0,1)$. Using the elementary inequality $1-e^{-x}\leq x^p$ for all $x>0$ and $p\in(0,1)$, we have
\begin{align*}
0\leq\E\left[\exp\left(-\mu I_{\D\setminus|w_1|^{1-\eta}\D}(\boldw)\right)-\exp\left(-\mu I_\D(\boldw)\right)\right]
&=\E\left[\exp\left(-\mu I_{\D\setminus|w_1|^{1-\eta}\D}(\boldw)\right)\left(1-\exp\left(-\mu I_{|w_1|^{1-\eta}\D}\right)\right)\right]\\
&\leq\E\left[1-\exp\left(-\mu I_{|w_1|^{1-\eta}\D}(\boldw)\right)\right]\\
&\leq\E[(I_{|w_1|^{1-\eta}\D}(\boldw))^p]=O(|w_1|^{(1-\eta)\xi(p)}),
\end{align*}
uniformly over $\Delta^r$. Similarly, we get the bound $\E[e^{-\mu I_{\D\setminus|w_1|^{1-\eta}\D}(0)}-e^{-\mu I_\D(0)}]=O(|w_1|^{(1-\eta)\xi(p)})$.

It remains to bound the contribution of the annulus $\D\setminus|w_1|^{1-\eta}\D$. In this annulus, we have the estimate $\prod_{j=1}^r\frac{|z-w_j|^{-\gamma^2}}{|z|^{-\gamma^2}}=1+O(|w_1|^\eta)$ uniformly as $w_1\to0$. In other words, the two measures in the definitions of $I_{\D\setminus|w_1|^{1-\eta}\D}(\boldw)$ and $I_{\D\setminus|w_1|^{1-\eta}\D}(0)$ are mutually absolutely continuous, with Radon--Nikodym derivative of the form $1+O(|w_1|^\eta)$ uniformly. Equivalently, there exists $C>0$ such that
\[|I_{\D\setminus|w_1|^{1-\eta}\D}(\boldw)-I_{\D\setminus|w_1|^{1-\eta}\D}(0)|\leq C|w_1|^\eta I_{\D\setminus|w_1|^{1-\eta}\D}(0)\leq C|w_1|^\eta I_\D(0).\]
It follows from the inequality $|e^{-x}-e^{-y}|\leq|x-y|^p$ valid for $x,y\geq0$ and $p\in(0,1)$, that 
\[|\E[e^{-\mu I_{\D\setminus|w_1|^{1-\eta}\D}(\boldw)}-e^{-\mu I_{\D\setminus|w_1|^{1-\eta}\D}(0)}]|\leq C|w_1|^{\eta p}\E[I_\D(0)^p].\]
For the choice of $p$ above, the moment is finite, and we are done.

\end{proof}


We can combine Lemma \ref{lem:dotsenko_fateev} and the last proposition to compute the residue of $\cI_{2,(1,1),(1,1)}$ at $\alpha_{2,1}$.

\begin{proof}[Proof of \eqref{eq:residue_I_two}]
In this proof, we use the following subsets of $\D^2$:
\[\cD_0:=\{|w_2|<|w_1|<e|w_2|\}\qquad\text{and}\qquad\cD_1:=\{|w_1|\geq e|w_2|\}.\]
Note that $\cD_1=\Delta^2$ from Proposition \ref{prop:fusion}, and recall the notation $\A_t=\{e^{-t}<|z|<1\}$. We will need to treat the cases $\gamma<\sqrt{2}$ and $\gamma>\sqrt{2}$ separately.

\emph{Case $\gamma>\sqrt{2}$}. In this case, $\alpha_{-2,1}=\alpha_{2,1}+2\gamma>Q$.

We first treat the region $\cD_0$. Recall the notation $\A_t=\{e^{-t}<|z|<1\}$. By Proposition \ref{prop:fusion} and scaling, we have 
\[\int_{r\A_2^2}\Psi_\alpha(w_1,w_2)\frac{|\d w_1|^2}{|w_1|^2}\frac{|\d w_2|^2}{|w_2|^2}=O(r^{-2\gamma\alpha-\gamma^2+\frac{1}{2}(Q-\alpha-2\gamma)^2})\]
as $r\to0$. At $\alpha=\alpha_{2,1}$, the exponent is $\frac{1}{2\gamma^2}(\gamma^2-2)^2>0$. By continuity in $\alpha$, this exponent is positive in a neighbourhood of $\alpha_{2,1}$. Using as before the covering $\cD_0\subset\cup_{n\in\N}e^{-n}\A_2^2$, we see that $\int_{\cD_0}\Psi_\alpha(w_1,w_2)\frac{|\d w_1|^2}{|w_1|^2}\frac{|\d w_2|^2}{|w_2|^2}$ is bounded by an absolutely convergent series.

Now, we deal with the region $\cD_1$. In this region, we have $|w_1-w_2|^{-\gamma^2}\leq((1-e^{-1})|w_1|)^{-\gamma^2}$. Combining with Proposition \ref{prop:fusion}, we then get
\begin{align*}
\int_{\cD_1}\Psi_\alpha(w_1,w_2)\frac{|\d w_1|^2}{|w_1|^2}\frac{|\d w_2|^2}{|w_2|^2}
&=\int_\D\int_{e^{-1}|w_1|\D}O(|w_1|^{-\gamma\alpha-\gamma^2+\xi}|w_2|^{-\gamma\alpha})\frac{|\d w_2|^2}{|w_2|^2}\frac{|\d w_1|^2}{|w_1|^2},
\end{align*}
for some $\xi>0$. This is integrable in a neighbourhood of $\alpha_{2,1}$, so we are done.

\emph{Case $\gamma<\sqrt{2}$}.

In this case, $\alpha_{-2,1}=\alpha_{2,1}+2\gamma<Q$.

By Proposition \ref{prop:fusion} and permutation symmetry of $(w_1,w_2)$, we have in $\cC'$ for $\alpha<\alpha_{2,1}$:

\begin{equation}\label{eq:some_residue}
\begin{aligned}
\cI_{2,(1,1),(1,1)}(\alpha)
&=2\Psi_{\alpha+2\gamma}\int_{\Delta^2}|w_1|^{-\gamma\alpha}|w_2|^{-\gamma\alpha}|w_1-w_2|^{-\gamma^2}\frac{|\d w_1|^2}{|w_1|^2}\frac{|\d w_2|^2}{|w_2|^2}\\
&\quad+2\int_{\Delta^2}\left(\Psi_\alpha(w_1,w_2)-|w_1|^{-\gamma\alpha}|w_2|^{-\gamma\alpha}|w_1-w_2|^{-\gamma^2}\Psi_{\alpha+2\gamma}\right)\frac{|\d w_1|^2}{|w_1|^2}\frac{|\d w_2|^2}{|w_2|^2}\\
&\quad+2\int_{\cD_0}\Psi_\alpha(w_1,w_2)\frac{|\d w_1|^2}{|w_1|^2}\frac{|\d w_2|^2}{|w_2|^2}.
\end{aligned}
\end{equation}
By the second item of Proposition \ref{prop:fusion}, the second line of the right-hand-side is uniformly integrable (hence analytic) up to $\alpha<\alpha_{2,1}+\xi$ for some $\xi>0$. Moreover, using that $\cD_0\subset\cup_{n\in\N}e^{-n}\A_2^2$ and the multifractal scaling relation (from the proof of Proposition \ref{prop:fusion}), we see that the third line is bounded by an absolutely convergent series, up to a neighbourhood of $\alpha_{2,1}$. Hence, the third line is regular at $\alpha_{2,1}$. Finally, we can argue similarly that the function $\alpha\mapsto\int_{\cD_0}|w_1|^{-\gamma\alpha}|w_2|^{-\gamma\alpha}|w_1-w_2|^{-\gamma^2}\frac{|\d w_1|^2}{|w_1|^2}\frac{|\d w_2|^2}{|w_2|^2}$ is bounded by an absolutely convergent series in a neighbourhood of $\alpha_{2,1}$, hence it is regular at $\alpha_{2,1}$. Thus, as far as the residue is concerned, we can replace the domain of integration by $\D^2$ (instead of $\Delta^2$) in the first line of the right-hand-side of \eqref{eq:some_residue}. Hence, by definition of $J_1$ \eqref{eq:def_J},
\[\underset{\alpha=\alpha_{2,1}}{\mathrm{Res}}\cI_{2,(1,1),(1,1)}(\alpha)=\Psi_{\alpha_{-2,1}}\underset{\alpha=\alpha_{2,1}}{\mathrm{Res}}\,J_1(\alpha).\]
Arguing as in the proof of Proposition \ref{prop:first_pole} shows the residue exists in $e^{-\beta c}\cD(\cQ)$ and the RHS also exists in $e^{-\beta c}\cD(\cQ)$. Hence the equality holds in $e^{-\beta c}\cD(\cQ)$. Lemma \ref{lem:dotsenko_fateev} gives the value of $\mathrm{Res}_{\alpha_{2,1}}\,J_1$, so we are done.
\end{proof}

		\subsubsection{Meromorphic continuation of the singular state}

  We recall the following derivative formula \cite[Section 3.1]{BW}: for $\boldw=(w_1,...,w_r)$ non-coinciding points and all $F\in\cC$:
\begin{equation}\label{eq:derivative_formula}
\begin{aligned}
\del_{w_1}\E[\Psi_\alpha(\boldw)F]
&=\left(\alpha\gamma\del_{w_1}G(w_1,0)+\gamma^2\sum_{j=2}^r\del_{w_j}G(w_1,w_j)\right)\E[\Psi_\alpha(\boldw)F]\\
&\quad+\E\left[\Psi_\alpha(\boldw)\nabla F(\del_{w_1}G_\del(w_1,\cdot\,))\right]\\
&\quad-\mu\gamma^2\int_\D\E\left[\Psi_\alpha(\boldw,w_{r+1})F\right]\del_{w_1}G(w_1,w_{r+1})|\d w_{r+1}|^2.
\end{aligned}
\end{equation}
 A similar formula holds for $\del_{\bar{w}_1}\Psi_\alpha(\boldw)$. 

\begin{proposition}\label{prop:mero}
For all $\alpha<\alpha_{1,2}$, we have
\begin{equation}\label{eq:mero_in_prop}
(\alpha-\alpha_{1,2})^2\cI_{1,(2),(2)}(\alpha)=-\frac{\mu\gamma^2}{4}\cI_{2,(1,1),(1,1)}(\alpha)+\cR_\alpha,
\end{equation}
where $\alpha\mapsto\cR_\alpha\in e^{-\beta c}\cD(\cQ)$ admits a meromorphic extension which is regular in a neighbourhood $\alpha_{2,1}$. Hence, the left-hand-side admits a meromorphic continuation in a neighbourhood of $\alpha_{2,1}$. In particular, \eqref{eq:mero} holds.
\end{proposition}

\begin{proof}
The proof relies on the iteration of the same three-step strategy: integration by parts, derivative formula, and the so-called ``symmetrisation trick". As we will see below, applying this strategy to an integral with singularities (such as $\int_\D\Psi_\alpha(w)\frac{|\d w|^2}{|w|^4}$) allows us to express it as a linear combination of integrals with milder singularities (for which we have a better range of absolute convergence). The algorithm stops when we can say that each integral is either regular at $\alpha_{2,1}$ or has a pole there (that we can compute).

Let us exemplify this strategy on the integral with singularity $|w_1|^{-4}$. For $\alpha<\alpha_{1,2}$, we have by integration by parts
\begin{align*}
\int_\D\Psi_\alpha(w_1)\frac{|\d w_1|^2}{|w_1|^4}
&=\int_\D\del_{w_1}\Psi_\alpha(w_1)\frac{|\d w_1|^2}{w_1\bar{w}_1^2}-\frac{1}{2i}\int_{\S^1}\Psi_\alpha(w_1)\frac{\d\bar{w}_1}{\bar{w}_1}.
\end{align*}
This formula is valid provided all the terms are absolutely convergent, which holds for $\alpha<\alpha_{1,2}$. We refer to \cite[Section 3.2]{BW} for the careful justification. Combining with the derivative formula \eqref{eq:derivative_formula}, we get
\begin{equation}\label{eq:exp_4}
\begin{aligned}
\frac{\gamma}{2}(\alpha-\alpha_{1,2})\int_\D\E[\Psi_\alpha(w_1)F]\frac{|\d w_1|^2}{|w_1|^4}
&=-\mu\gamma^2\int_{\D^2}\E[\Psi_\alpha(w_1,w_2)F]\del_{w_1}G(w_1,w_2)\frac{|\d w_1|^2}{w_1\bar{w}_1^2}|\d w_2|^2\\
&\quad+\int_\D\E[\Psi_\alpha(w_1)\nabla F(\del_{w_1}G_\del(w_1,\cdot\,))]\frac{|\d w_1|^2}{w_1\bar{w}_1^2}\\
&\quad-\frac{1}{2i}\int_{\S^1}\E[\Psi_\alpha(w_1)F]\frac{\d\bar{w_1}}{\bar{w}_1}\\
&=:\cI_1+\cI_2+\cI_3,
\end{aligned}
\end{equation}
where $\cI_j$ is the $j^\text{th}$ line of the last equality. The integral $\cI_3$ is regular for $\alpha<0$. To treat $\cI_2$, we use again the integration by parts and the derivative formulas to get for all $F\in\cC$:
\begin{align*}
\frac{\gamma}{2}(\alpha-\alpha_{1,2})\int_\D\E[\Psi_\alpha(w_1)F]\frac{|\d w_1|^2}{w_1\bar{w}_1^2}
&=-\mu\gamma^2\int_{\D^2}\E[\Psi_\alpha(w_1,w_2)F]\del_{w_1}G(w_1,w_2)\frac{|\d w_1|^2}{|w_1|^2}|\d w_2|^2\\
&\quad+\int_\D\E[\Psi_\alpha(w_1)\nabla F(\del_{w_1}G_\del(w_1,\cdot\,))]\frac{|\d w_1|^2}{|w_1|^2}\\
&\quad-\frac{1}{2i}\int_{\S^1}\E[\Psi_\alpha(w_1)F]\d\bar{w_1}.
\end{align*}
Applying this formula with $F$ replaced by $\nabla F(\del_{w_1}G_\del(w_1,\cdot\,))\in\cC$, we see that $\cI_2$ is also regular for $\alpha<\alpha_{2,1}+\frac{1}{2\gamma}$ (by \cite[Proposition 3.2]{BW}). 

Hence, the only term with a possible pole at $\alpha_{2,1}$ is $\cI_1$. Now, we introduce the symmetrisation trick. Note that $\del_{w_1}G(w_1,w_2)=-\frac{1}{2(w_1-w_2)}$ is antisymmetric under the permutation of $w_1$ and $w_2$, while $\Psi_\alpha(w_1,w_2)$ is permutation invariant. Hence,
\[\int_{\D^2}\Psi_\alpha(w_1,w_2)\frac{1}{w_2-w_1}\frac{|\d w_1|^2}{w_1\bar{w}_1^2}|\d w_2|^2=\frac{1}{2}\int_{\D^2}\Psi_\alpha(w_1,w_2)\frac{1}{w_2-w_1}\left(\frac{1}{w_1\bar{w}_1^2}-\frac{1}{w_2\bar{w}_2^2}\right)|\d w_1|^2|\d w_2|^2.\]
 
Now, observe that
\begin{equation}\label{eq:symmetrise}
\begin{aligned}
\frac{1}{w_2-w_1}\left(\frac{1}{w_1\bar{w}_1^2}-\frac{1}{w_2\bar{w}_2^2}\right)
&=\frac{1}{w_2-w_1}\frac{w_2\bar{w}_2^2-w_1\bar{w}_1^2}{w_1\bar{w}_1^2w_2\bar{w}_2^2}\\
&=\frac{1}{w_1\bar{w}_1^2w_2}+\frac{\bar{w}_2-\bar{w}_1}{w_2-w_1}\left(\frac{1}{\bar{w}_1w_2\bar{w}_2^2}+\frac{1}{\bar{w}_1^2w_2\bar{w}_2}\right).
\end{aligned}
\end{equation}
This gives us three new types of singularities to study. We start with the one in $\frac{1}{w_1\bar w_1^2w_2}$. Applying again the integration by parts and the derivative formulas, we have
\begin{equation}\label{eq:Is}
\begin{aligned}
\frac{\gamma}{2}(\alpha-\alpha_{1,2})\int_{\D^2}\E[\Psi_\alpha(w_1,w_2)F]\frac{|\d w_1|^2}{w_1\bar{w}_1^2}\frac{|\d w_2|^2}{w_2}
&=\gamma^2\int_{\D^2}\E[\Psi_\alpha(w_1,w_2)F]\del_{\bar{w}_1}G(w_1,w_2)\frac{|\d w_1|^2}{|w_1|^2}\frac{|\d w_2|^2}{w_2}\\
&\quad-\mu\gamma^2\int_{\D^3}\E[\Psi_\alpha(w_1,w_2,w_3)F]\del_{\bar{w}_1}G(w_1,w_3)\frac{|\d w_1|^2}{|w_1|^2}\frac{|\d w_2|^2}{w_2}|\d w_3|^2\\
&\quad+\int_{\D^2}\E[\Psi_\alpha(w_1,w_2)\nabla F(\del_{\bar{w}_1}G_\del(w_1,\cdot\,))]\frac{|\d w_1|^2}{|w_1|^2}\frac{|\d w_2|^2}{w_2}\\
&\quad-\frac{1}{2i}\int_{\S^1}\int_\D\E[\Psi_\alpha(w_1,w_2)F]\frac{|\d w_2|^2}{w_2}\d w_1\\
&=:\cI_4+\cI_5+\cI_6+\cI_7.
\end{aligned}
\end{equation}
The integrals $\cI_6$ and $\cI_7$ are absolutely convergent and analytic for $\alpha<\alpha_{2,1}+\frac{1}{2\gamma}$. For $\cI_4$, we use the identity $\frac{1}{\bar{w}_2-\bar{w}_1}(\frac{1}{\bar{w_1}}-\frac{1}{\bar{w_2}})=\frac{1}{\bar w_1\bar w_2}$ and apply the symmetrisation trick, to see that 
\[\cI_4=\frac{\gamma^2}{4}\int_{\D^2}\E[\Psi_\alpha(w_1,w_2)F]\frac{|\d w_1|^2}{|w_1|^2}\frac{|\d w_2|^2}{|w_2|^2}=\frac{\gamma^2}{4}\E[\cI_{2,(1,1),(1,1)(\alpha)}F].\]
This gives the leading term in the RHS of \eqref{eq:mero_in_prop}. It remains to show that all other terms are analytic in a neighbourhood of $\alpha_{2,1}$.

Now, we postpone the analysis of $\cI_5$ to the end of the proof, and start to analyse the singularities with the factor $\frac{\bar{w}_2-\bar{w}_1}{w_2-w_1}$ appearing in \eqref{eq:symmetrise}. We only treat the singularity $\frac{\bar{w}_2-\bar{w}_1}{w_2-w_1}\frac{1}{w_1\bar{w}_1^2\bar{w}_2}$ since the other one can be treated similarly. Since $|\frac{\bar{w}_2-\bar{w}_1}{w_2-w_1}|=1$, it is sufficient to show that $\cI_{2,(1,0),(2,1)}$ is absolutely convergent for all $\alpha$ in a neighbourhood of $\alpha_{2,1}$. We will need to treat the cases $\gamma<\sqrt{2}$ and $\gamma>\sqrt{2}$ separately. The method is similar to the proof of \eqref{eq:residue_I_two} given at the end of Section \ref{subsubsec:fusion}.

\emph{Case $\gamma<\sqrt{2}$}. In this case, we can write

\begin{equation}\label{eq:I}
\begin{aligned}
\cI_{2,(1,0),(2,1)}(\alpha)
&=\Psi_{\alpha+2\gamma}\int_{\D^2}|w_1|^{-\gamma\alpha}|w_2|^{-\gamma\alpha}|w_1-w_2|^{-\gamma^2}\frac{|\d w_1|^2}{w_1\bar{w}_1^2}\frac{|\d w_2|^2}{\bar{w}_2}\\
&\quad+\int_{\D^2}\left(\Psi_\alpha(w_1,w_2)-|w_1|^{-\gamma\alpha}|w_2|^{-\gamma\alpha}|w_1-w_2|^{-\gamma^2}\Psi_{\alpha+2\gamma}\right)\frac{|\d w_1|^2}{w_1\bar{w}_1^2}\frac{|\d w_2|^2}{\bar{w}_2}.
\end{aligned}
\end{equation}
By Lemma \ref{lem:dotsenko_fateev}, the first integral has a meromorphic continuation which is regular at $\alpha=\alpha_{2,1}$. For the second integral, Proposition \ref{prop:fusion} tells us that the integrand is $O(|w_1|^{-\gamma\alpha}|w_2|^{-\gamma\alpha}|w_1-w_2|^{-\gamma^2}(|w_1|\vee|w_2|)^\xi)$ for some $\xi>0$, which is integrable in a neighbourhood of $\alpha_{2,1}$. Hence, $\cI_{2,(1,0),(2,1)}$ is absolutely convergent and analytic in a neighbourhood of $\alpha_{2,1}$.

\emph{Case $\gamma>\sqrt{2}$}. We decompose $\D^2$ into the regions
\[\cD_0:=\{|w_1|\vee|w_2|\leq e|w_1|\wedge|w_2|\};\qquad\cD_1:=\{|w_1|\leq e^{-1}|w_2|\};\qquad\cD_2:=\{|w_2|\leq e^{-1}|w_1|\}.\]
We start with the region $\cD_0$. First, we look at the case where both insertions are in the annulus $r\A_2=\{re^{-2}<|z|<r\}$ for some $r>0$. By item \ref{item:geq} of Proposition \ref{prop:fusion}, we have 
\[\left|\int_{r\A_2^2}\E[\Psi_\alpha(w_1,w_2)F]\frac{\bar{w}_2-\bar{w}_1}{w_2-w_1}\frac{|\d w_1|^2}{w_1\bar{w}_1^2}\frac{|\d w_2|^2}{\bar{w}_2}\right|=O(r^{-2\gamma\alpha-\gamma^2+\frac{1}{2}(Q-\alpha-2\gamma)^2}).\]
Observe that $-2\gamma\alpha_{2,1}-\gamma^2+\frac{1}{2}(Q-\alpha_{2,1}-2\gamma)^2=\frac{1}{2}(\gamma-\frac{2}{\gamma})^2>0$, so that the exponent is positive in a neighbourhood of $\alpha_{2,1}$. Using the covering $\cD_0\subset\cup_{n\in\N}e^{-n}\A_2^2$, the integral $\int_{\cD_0}\E[\Psi_\alpha(w_1,w_2)F]\frac{\bar{w}_2-\bar{w}_1}{w_2-w_1}\frac{|\d w_1|^2}{w_1\bar{w}_1^2}\frac{|\d w_2|^2}{\bar{w}_2}$ is then bounded by an absolutely convergent series, so that it converges and is analytic in a neighbourhood of $\alpha_{2,1}$.

We turn to the contribution of $\cD_1$. In this region, we have $|w_1-w_2|^{-\gamma^2}\leq((1-e^{-1})|w_2|)^{-\gamma^2}$. Using again item \ref{item:geq} of Proposition \ref{prop:fusion}, we get
\begin{align*}
\left|\int_{\cD_1}\E[\Psi_\alpha(w_1,w_2)F]\frac{|\d w_1|^2}{w_1\bar{w}_1^2}\frac{|\d w_2|^2}{\bar{w}_2}\right|
&\leq C\int_\D\int_{e^{-1}|w_2|\D}|w_2|^{-\gamma\alpha-\gamma^2+\frac{1}{2}(Q-\alpha-2\gamma)^2}|w_1|^{-\gamma\alpha}\frac{|\d w_1|^2}{|w_1|^3}\frac{|\d w_2|^2}{|w_2|}\\
&\leq C\int_\D|w_2|^{-2\gamma\alpha-\gamma^2+\frac{1}{2}(Q-\alpha-2\gamma)^2}\frac{|\d w_2|^2}{|w_2|^2}.
\end{align*}
In the first line, the integral in $w_1$ is absolutely convergent in a neighbourhood of $\alpha_{2,1}$ since $\frac{\gamma^2}{2}-3>-2$. Then we have already seen that the exponent in $w_2$ in the last integral is positive in that region. Hence, the contribution of $\cD_1$ is absolutely convergent and analytic in a neighbourhood of $\alpha_{2,1}$. The last region $\cD_2$ is treated similarly. This concludes the proof that $\cI_{2,(1,0),(2,1)}$ is analytic in a neighbourhood of $\alpha_{2,1}$ in the case $\gamma>\sqrt{2}$.\medskip

\textbf{Regularity of $\cI_5$.} Recall the definition of $\cI_5$ in \eqref{eq:Is}.

 First, we do the symmetrisation trick by using the identity $\frac{1}{\bar{w}_3-\bar{w}_1}(\frac{1}{|w_1|^2}-\frac{1}{|w_3|^2})=\frac{1}{|w_1|^2\bar{w}_3}+\frac{w_3-w_1}{\bar w_3-\bar w_1}\frac{1}{w_1|w_3|^2}$. This shows that $\cI_5$ is expressed as a linear combination where the singularities are bounded by $\frac{1}{|w_1|^2|w_2||w_3|}$ (and permutations). In other words, it suffices to show that $\cI_{3,(2,1,1),(0,0,0)}$ is absolutely convergent in a neighbourhood of $\alpha_{2,1}$. We will need to distinguish the cases $\gamma<1$, $\gamma\in[1,\sqrt{2})$, and $\gamma\in(\sqrt{2},2)$.\smallskip

\emph{Case $\gamma<1$}. In this case, \cite[Proposition 3.2]{BW} says that $\cI_5$ is uniformly integrable up to $\alpha<-\gamma+\frac{2}{3\gamma}$. In the range $\gamma<1$, we have $\alpha_{2,1}=-\frac{\gamma}{2}<-\gamma+\frac{2}{3\gamma}$, so in particular $\cI_5$ is analytic in a neighbourhood of $\alpha_{2,1}$.
%

\emph{Case $\gamma\in[1,\sqrt{2})$}. In this case, we have $\alpha_{2,1}+3\gamma\geq Q$, and $\alpha_{2,1}+2\gamma<Q$. 

By permutation symmetry, it is sufficient to treat the singularity $\frac{1}{|w_1|^2|w_2||w_3|}$ only. We define the following subregions of $\D^3$
\begin{align*}
&\cD_0:=\{\max\{|w_1|,|w_2|,|w_3|\}\leq e\min\{|w_1|,|w_2|,|w_3|\}\};\\
&\cD_1:=\{|w_1|\leq e^{-1}|w_2|,\,|w_1|\leq e^{-1}|w_3|\}\qquad\text{and}\qquad\cD_1':=\{|w_1|\geq e|w_2|,\,|w_1|\geq e|w_3|\}.
\end{align*}
To treat $\cD_0$, we look at the case where all three insertions are in the annulus $r\A_2$ for some $r\in(0,1)$. By the first item of Proposition \ref{prop:fusion}, we have
\[\int_{r\A_3^3}|\E[\Psi_\alpha(w_1,w_2,w_3)F]|\frac{|\d w_1|^2}{|w_1|^2}\frac{|\d w_2|^2}{|w_2|}\frac{|\d w_3|^2}{|w_3|}=O(r^{-3\gamma\alpha-3\gamma^2+\frac{1}{2}(Q-\alpha-3\gamma)^2+2}).\]
At $\alpha=\alpha_{2,1}$, we have
\begin{equation}\label{eq:exponent}
-3\gamma\alpha_{2,1}-3\gamma^2+\frac{1}{2}(\alpha_{2,1}+3\gamma-Q)^2=-\frac{3\gamma^2}{2}+2(\frac{1}{\gamma}-\gamma)^2=\frac{\gamma^2}{2}+\frac{2}{\gamma^2}-4=(\frac{\gamma}{\sqrt{2}}-\frac{\sqrt{2}}{\gamma})^2-2>-2,
\end{equation}
so that the exponent in the previous display is positive around $\alpha_{2,1}$. Using the cover $\cD_0\subset\cup_{n\in\N}e^{-n}\A_3^3$, we get that the integral $\int_{\cD_0}\Psi_\alpha(w_1,w_2,w_3)\frac{|\d w_1|^2}{|w_1|^2}\frac{|\d w_2|^2}{|w_2|}\frac{|\d w_3|^2}{|w_3|}$ is bounded by an absolutely convergent series. 

Let us turn to $\cD_1$. In this region, we have $|w_1-w_2|^{-\gamma^2}\leq((1-e^{-1})|w_2|)^{-\gamma^2}$. By the first item of Proposition~\ref{prop:fusion}, we have
\begin{align*}
&\int_{\cD_1}|\E[\Psi_\alpha(w_1,w_2,w_3)F]|\frac{|\d w_1|^2}{|w_1|^2}\frac{|\d w_2|^2}{|w_2|}\frac{|\d w_3|^2}{|w_3|}\\
&\quad\leq C\int_{\D^2}\int_{e^{-1}(|w_2|\wedge|w_3|)\D}|w_1|^{-\gamma\alpha}|w_2|^{-\gamma\alpha-\gamma^2}|w_3|^{-\gamma\alpha-\gamma^2}|w_2-w_3|^{-\gamma^2}(|w_2|\vee|w_3|)^{\frac{1}{2}(Q-\alpha-3\gamma)^2}\frac{|\d w_1|^2}{|w_1|^2}\frac{|\d w_2|^2}{|w_2|}\frac{|\d w_3|^2}{|w_3|}\\
&\quad\leq C\int_{\D^2}|w_2|^{-\gamma\alpha-\gamma^2}|w_3|^{-\gamma\alpha-\gamma^2}|w_2-w_3|^{-\gamma^2}(|w_2|\wedge|w_3|)^{-\gamma\alpha}(|w_2|\vee|w_3|)^{\frac{1}{2}(Q-\alpha-3\gamma)^2}\frac{|\d w_2|^2}{|w_2|}\frac{|\d w_3|^2}{|w_3|}.
\end{align*}
Equation \eqref{eq:exponent} shows that this is integrable in a neighbourhood of $\alpha_{2,1}$. In the region $\cD_1'$, we have $|w_1-w_2|^{-\gamma^2}\leq ((1-e^{-1})|w_1|^{-\gamma^2}$, and the first item of Proposition \ref{prop:fusion} gives the bound
\begin{align*}
&\int_{\cD_1'}|\E[\Psi_\alpha(w_1,w_2,w_3)F]|\frac{|\d w_1|^2}{|w_1|^2}\frac{|\d w_2|^2}{|w_2|}\frac{|\d w_3|^2}{|w_3|}\\
&\quad\leq C\int_\D\int_{e^{-1}|w_1|\D^2}|w_1|^{-\gamma\alpha-2\gamma^2+\frac{1}{2}(Q-\alpha-3\gamma)^2}|w_2|^{-\gamma\alpha-\gamma^2}|w_3|^{-\gamma\alpha-\gamma^2}|w_2-w_3|^{-\gamma^2}\frac{|\d w_1|^2}{|w_1|^2}\frac{|\d w_2|^2}{|w_2|}\frac{|\d w_3|^2}{|w_3|}\\
&\quad\leq C\int_\D|w_1|^{-3\gamma\alpha-3\gamma^2+\frac{1}{2}(Q-3\gamma-\alpha)^2+2}\frac{|\d w_1|^2}{|w_1|^2}.
\end{align*}
Again, we arrive at the conclusion that the integral converges uniformly in a neighbourhood of $\alpha_{2,1}$, and is regular there. We can perform the same bounds in the regions
\[\cD_2:=\{|w_2|\leq e^{-1}|w_1|,\,|w_2|\leq e^{-1}|w_3|\}\qquad\text{and}\qquad\cD_2':=\{|w_2|\geq e|w_1|,\,|w_2|\geq e|w_3|\},\]
and we arrive at the same conclusion. By the $w_2\leftrightarrow w_3$ symmetry, this covers all possible regions of $\D^3$.\smallskip

\emph{Case $\gamma\in(\sqrt{2},2)$}. In this case, $\alpha_{2,1}+2\gamma>Q$. The estimate in the region $\cD_0$ from the previous case is still valid, and the same conclusion holds for this region. We need to further decompose $\cD_1$ into subregions. We introduce
\[\cD_{10}:=\{|w_1|\leq e^{-1}|w_2|,\,|w_2|\leq|w_3|\leq e|w_2|\};\qquad\cD_{11}:=\{|w_1|\leq e^{-1}|w_2|\leq e^{-2}|w_3|\}.\]
By Proposition \ref{prop:fusion} and scaling,
\begin{align*}
\int_{r\A_2^2}\int_{e^{-1}r\D}|\E[\Psi_\alpha(w_1,w_2,w_3)F]|\frac{|\d w_1|^2}{|w_1|^2}\frac{|\d w_2|^2}{|w_2|}=O(r^{-3\gamma\alpha-3\gamma^2+\frac{1}{2}(Q-\alpha-3\gamma)^2+2}),
\end{align*}
and we have seen that the exponent is positive in a neighbourhood of $\alpha_{2,1}$. This allows us to bound the integral over $\cD_{10}$ by an absolutely convergent series, so that this integral is analytic in a neighbourhood of $\alpha_{2,1}$. In the region $\cD_{11}$, we have $\Psi_\alpha(w_1,w_2,w_3)=O(|w_1|^{-\gamma\alpha}|w_2|^{-\gamma\alpha-\gamma^2+\frac{1}{2}(Q-\alpha-2\gamma)^2}|w_3|^{-\gamma\alpha-\frac{3\gamma^2}{2}})$, and this is integrable on $(\D^3,\frac{|\d w_1|^2}{|w_1|^2}\frac{|\d w_2|^2}{|w_2|}\frac{|\d w_3|^2}{|w_3|})$ in a neighbourhood of $\alpha_{2,1}$. We can decompose similarly the regions $\cD_1'$, $\cD_2$, $\cD_2'$ into subregions, and we get absolutely convergent integrals each time in a neighbourhood of $\alpha_{2,1}$. All these steps are similar to the previous one, and we leave the details to the reader. In the end, we have shown that $\cI_5$ is analytic in a neighbourhood of $\alpha_{2,1}$, for all $\gamma\in(0,2)\setminus\{\sqrt{2}\}$.

This concludes the proof of \ref{eq:mero_in_prop}. To end the proof of Proposition \ref{prop:mero}, it remains to see how this implies~\eqref{eq:mero}. Combining Proposition \ref{prop:mero} with Proposition \ref{prop:expression_bulk}, we obtain
\begin{align*}
(\alpha-\alpha_{1,2})^2\cP_\alpha(4|\varphi_2|^2-1)
&=-\mu\frac{\gamma^2}{4}(\alpha-\alpha_{1,2})^2\cI_{1,(2),(2)}(\alpha)+\mu^2\frac{\gamma^2}{4}(\alpha-\alpha_{1,2})^2\cI_{2,(2,0),(0,2)}(\alpha)+\cR_\alpha\\
&=\mu^2\frac{\gamma^4}{16}\cI_{2,(1,1),(1,1)}(\alpha)+\mu^2\frac{\gamma^2}{4}(\alpha-\alpha_{1,2})^2\cI_{2,(2,0),(0,2)}(\alpha)+\cR_\alpha,
\end{align*}
where $\cR_\alpha$ denotes an analytic term in a neighbourhood of $\alpha_{2,1}$, which may vary from line to line. It remains to show that $\cI_{2,(2,0),(0,2)}$ is regular at $\alpha_{2,1}$; the proof is identical to that of the regularity of $\cI_{2,(1,0),(2,1)}$ just above, and we omit it. This shows that 
\[(\alpha-\alpha_{2,1})^2\cP_\alpha(4|\varphi_2|^2-1)=\mu^2\frac{\gamma^4}{16}\cI_{2,(1,1),(1,1)}(\alpha)+\cR_\alpha,\]
for some $\cR_\alpha$ which is regular at $\alpha_{2,1}$, and concludes the proof of \eqref{eq:mero}. 
\end{proof}

\section{Boundary HEM}\label{sec:boundary}

In this section, we prove the boundary version of the HEMs (Theorem \ref{thm:boundary_hem}). This section and the notations therein are independent from Section \ref{sec:bulk}; for instance, the space $\cF$ and measure $\P_{\S^1}$ introduced below are different from (although very similar to) those in Section \ref{sec:bulk}. The introductory Section~\ref{subsec:boundary_setup} gives some background on boundary LCFT. The remaining sections follow the structure of the proof of Theorem \ref{thm:bulk}. All the technical difficulties are already present in the bulk case, and the differences are mostly computational. We will use most of the notations of Section \ref{sec:bulk}, but all the notations from this section refer to boundary LCFT.

	\subsection{Setup and background}\label{subsec:boundary_setup}

		\subsubsection{Free field modules}
Let $\cF:=\R[(\varphi_n)_{n\geq1}]$ be the space of polynomials in countably many real variables~$\varphi_n$. The constant function $\ind$ is the \emph{vacuum vector}. Let $\P_{\S^1}$ be the law of a $\log$-correlated Gaussian field $\varphi$ on $\S^1\cap\H$:
\[\E[\varphi(e^{i\theta})\varphi(e^{i\theta'})]=\log\frac{1}{|e^{i\theta}-e^{i\theta'}|}+\log\frac{1}{|e^{i\theta}-e^{-i\theta'}|}.\]
The expansion in Fourier modes reads
\[\varphi=2\sum_{n=1}^\infty\varphi_n\Re(e_n).\]
where we recall that $e_n(e^{i\theta})=e^{ni\theta}$ is the standard basis. Under $\P_{\S^1}$, $(\varphi_n)_{n\geq1}$ is a sequence of independent real Gaussians $\cN(0,\frac{1}{2n})$. The harmonic extension of $\varphi$ to $\D\cap\H$ is
\[P\varphi(z)=2\sum_{n=1}^\infty\varphi_n\Re(z^n).\]
Its covariance kernel is
\[\E[P\varphi(z)P\varphi(w)]=\log\frac{1}{|1-z\bar{w}|}+\log\frac{1}{|1-zw|}=:G_\del(z,w),\qquad\forall z,w\in\D\cap\H.\]

The space $\cF$ is a dense subspace of $L^2(\P_{\S^1})$. The \emph{boundary Liouville Hilbert space} is 
\[\cH:=L^2(\d c\otimes\P_{\S^1}),\]
where $\d c$ is Lebesgue measure on $\R$. Samples of $\d c\otimes\P_{\S^1}$ are written $c+\varphi$, with $c$ being the zero mode of the field. We define a dense subspace $\cC$ of $\cH$ as the subspace of functionals $F\in\cH$ such that there exists $N\in\N^*$ and $f\in\cC^\infty(\R^{N+1})$, such that $F(c+\varphi)=f(c,\varphi_1,...,\varphi_N)$ holds $\d c\otimes\P_{\S^1}$-a.e., and $f$ and all its derivatives are compactly supported in $c$ and have at most exponential growth in the other variables. The space $\cC$ comes with a natural Fr\'echet topology, and we denote by $\cC'$ the space of continuous linear forms on $\cC$.

 Let $\alpha\in\C$. We have a representation of the Heisenberg algebra $(\bA_n^\alpha)_{n\in\Z}$ acting as densely defined operators on $L^2(\P_{\S^1})$, given for $n>0$ by
\begin{align*}
&\bA_n^\alpha=\frac{i}{2}\del_n;&\bA_{-n}^\alpha=\frac{i}{2}(\del_n-2n\varphi_n);&&\bA_0^\alpha=\frac{i}{2}\alpha;\\
\end{align*}
Here, $\del_n=\del_{\varphi_n}$ is the real derivative in the direction $\varphi_n$. For $n\neq0$, we have the hermiticity relations $(\bA_n^\alpha)^*=\bA_{-n}^\alpha$ on $L^2(\P_{\S^1})$.

The \emph{Sugawara construction} is a family of representations $(\bL_n^{0,\alpha})_{n\in\Z}$ of the Virasoro algebra on $L^2(\P_{\S^1})$, indexed by $\alpha\in\C$. These operators are the following quadratic expression in the Heisenberg algebra ($n\neq0$)
\begin{align*}
&\bL_n^{0,\alpha}:=i(\alpha-(n+1)Q)\bA_n+\sum_{m\neq\{0,n\}}\bA_{n-m}\bA_m;&\bL_0^{0,\alpha}:=\Delta_\alpha+2\sum_{m=1}^\infty\bA_{-m}\bA_m.
\end{align*}
This representation satisfies the hermiticity relations $(\bL_n^{0,\alpha})^*=\bL_{-n}^{0,2Q-\bar{\alpha}}$ on $L^2(\P_{\S^1})$. Given a Young diagram $\nu=(\nu_1,...,\nu_\ell)$, we set $\bL_{-\nu}^{0,\alpha}=\bL_{-\nu_\ell}^{0,\alpha}...\bL_{-\nu_1}^{0,\alpha}$. The descendant state 
\[\cQ_{\alpha,\nu}:=\bL_{-\nu}^{0,\alpha}\ind\in\cF\]
is a polynomial of level $|\nu|$.

Let $\cV_\alpha^0$ be the $(\bL_n^{0,\alpha})_{n\in\Z}$ highest-weight representation obtained by acting with the Virasoro operators on the vacuum vector, i.e. 
\[\cV_\alpha^0:=\mathrm{span}\{\bL_{-\nu}^{0,\alpha}\ind|,\nu\in\cT\}\subset\cF.\]
 We also define $\cV_\alpha^{0,N}:=\cV_\alpha^0\cap\cF_N$. The module $\cV_\alpha^0$ has central charge $c_\mathrm{L}=1+6Q^2$ and highest-weight $\Delta_\alpha=\frac{\alpha}{2}(Q-\frac{\alpha}{2})$. If $\alpha\not\in kac$, it is known that $\cV_\alpha^0$ is irreducible and Verma \cite[Lecture 8]{KacRaina_Bombay} (see also \cite[Appendix B]{BW}), hence $\cV_\alpha^0\simeq\cF$). On the other hand, if $\alpha\in kac^-$, $\cV_{2Q-\alpha}^0$ is Verma (hence $\cV_{2Q-\alpha}^0\simeq\cF$), and $\cV_\alpha^0$ is the irreducible quotient of the Verma by the maximal proper submodule. The linear map
\begin{equation*}
\Phi_\alpha^0:\left\lbrace\begin{aligned}
&\cV_{2Q-\alpha}^0\simeq\cF&&\to\cV_\alpha^0\\
&\bL_{-\nu}^{0,2Q-\alpha}\ind&&\mapsto\bL_{-\nu}^{0,\alpha}\ind
\end{aligned}\right.
\end{equation*}
implements the canonical projection from the Verma to $\cV_\alpha^0$, and $\cV_\alpha^0=\mathrm{ran}\,\Phi_\alpha^0\simeq\cF/\ker\Phi_\alpha^0$. It maps $\cQ_{2Q-\alpha,\nu}$ to $\cQ_{\alpha,\nu}$.

In the sequel, we will write
\begin{align*}
&\bS^0_\alpha:=\alpha^2\bL_{-2}^{0,\alpha}+(\bL_{-1}^{0,\alpha})^2.
\end{align*}
By the same computation as Lemma \ref{lem:ff}, we record the expression for $\bS^0_\alpha\ind$, which gives the singular vector at level 2.
\begin{lemma}\label{lem:ff_boundary}
We have 
\[\bS^0_\alpha\ind=2\alpha(\alpha-\alpha_{1,2})(\alpha-\alpha_{2,1})\varphi_2.\]
\end{lemma}

		\subsubsection{Semigroups and Poisson operator}
Let $X_\D$ be the GFF in $\H\cap\D$ with Dirichlet boundary conditions on $\S^1\cap\H$ and Neumann boudary conditions on $(-1,1)$. Explicitly, $X_\D$ is the Gaussian field with covariance:
\[\E[X_\D(z)X_\D(w)]=\log\left|\frac{1-z\bar{w}}{z-w}\right|+\log\left|\frac{1-zw}{z-\bar{w}}\right|=:G_\D(z,w).\]
We take this free field to be independent of $c+\varphi$. The covariance kernel of $X=X_\D+P\varphi$ is:
\[\E[X(z)X(w)]=\log\frac{1}{|z-w|}+\log\frac{1}{|z-\bar{w}|}=:G(z,w).\]
One can construct two GMCs from $X$, a bulk GMC in $\D\cap\H$ and a boundary GMC in $\I=(-1,1)$:
\begin{align*}
&\d M_\gamma(z):=\underset{\epsilon\to0}\lim\,\epsilon^\frac{\gamma^2}{2}e^{\gamma X_\epsilon(z)}|\d z|^2=\underset{\epsilon\to0}\lim\,e^{\gamma X_\epsilon(z)-\frac{\gamma^2}{2}\E[X_\epsilon(z)^2]}\frac{|\d z|^2}{(2\Im(z))^\frac{\gamma^2}{2}};\\
&\d L_\gamma(x):=\underset{\epsilon\to0}\lim\,\epsilon^\frac{\gamma^2}{4}e^{\frac{\gamma}{2}X_\epsilon(x)}\d x,
\end{align*}
where we use the same regularisation as in \cite[Section 2.9]{GRVW}: we use circle averages in the bulk, and semi-circle averages on $(-1,1)$. We will sometimes abuse notations by writing $\d M_\gamma(z)=e^{\gamma X(z)-\frac{\gamma^2}{2}\E[X^2(z)]}\frac{|\d z|^2}{(2\Im(z))^\frac{\gamma^2}{2}}$.

As in the bulk case, one can define two one-parameter semigroups of operators on $\cH$. The \emph{free field semigroup} is the semigroup generated by $-\bL_0^0$. Currently, the generator of the Liouville semigroup is well-understood only if $\gamma<\sqrt{2}$ or $\mu=0$ \cite{GRVW}, but the semigroup itself is well-defined for all values of the parameters. Explictly, these two semigroups have the following expression \cite[Section 7.2]{GRVW} (see in particular Proposition 7.3 therein):
\begin{equation}\label{eq:def_semigroups_boundary}
\begin{aligned}
&S_t^0F(c+\varphi)=e^{-\frac{Q^2}{4}t}\E_\varphi\left[F(c+X(e^{-t}\cdot\,))\right]\\
&S_tF(c+\varphi)=e^{-\frac{Q^2}{4}t}\E_\varphi\left[F(c+X(e^{-t}\cdot\,))e^{-\mu e^{\gamma c}\int_{\A_t\cap\H}\frac{\d M_\gamma(z)}{|z|^{\gamma Q}}}e^{-\mu_\mathrm{L}e^{\frac{\gamma}{2}c}\int_{\I_t^-}\frac{\d L_\gamma(x)}{|x|^{\frac{\gamma Q}{2}}}}e^{-\mu_\mathrm{R}e^{\frac{\gamma}{2}c}\int_{\I_t^+}\frac{\d L_\gamma(x)}{|x|^{\frac{\gamma Q}{2}}}}\right],
\end{aligned}
\end{equation}
where we defined $\I_t^+:=(e^{-t},1)$ and $\I_t^-:=(-1,-e^{-t})$. We also set $\I^+:=(0,1)$ and $\I^-:=(-1,0)$.
%
 
We wish to construct a Poisson operator in a similar way as in the bulk case. For $\Re(\alpha)<Q$ and any $\nu\in\cT$, we define (if the limit exists in the weighted space $e^{-\beta c}\cH$ for some $\beta\in\R$)
\begin{equation}\label{eq:def_descendant_boundary}
\begin{aligned}
&\cP_\alpha(\cQ_{\alpha,\nu}):=\underset{t\to\infty}\lim\,e^{t(\Delta_\alpha+|\nu|)}S_t(e^{\frac{1}{2}(\alpha-Q)c}\cQ_{\alpha,\nu})\\
&\Phi_\alpha(\chi):=\cP_\alpha(\Phi_\alpha^0(\cQ_{\alpha,\nu})).
\end{aligned}
\end{equation}

Applied to the constant function $\chi=1$ and for $\alpha<Q$, the limiting function is seen to be given by
\[\Psi_\alpha^\del:=\Phi_\alpha(\ind)=e^{\frac{1}{2}(\alpha-Q)c}\E_\varphi\left[e^{-\mu\int_{\D\cap\H}\frac{\d M_\gamma(z)}{|z|^{\gamma\alpha}}}e^{-\int_\I\mu_\del(x)\frac{\d L_\gamma(x)}{|x|^{\frac{\gamma\alpha}{2}}}}\right],\]
where we have defined the measurable function on $\I=(-1,1)$:
\[\mu_\del(x)=\mu_\mathrm{L}\ind_{\{x<0\}}+\mu_\mathrm{R}\ind_{\{x>0\}}.\]
The goal of this section is to study the state $\Phi_\alpha(\bS_{2Q-\alpha}^0\ind)$. We will find a probabilistic expression for $\alpha<\alpha_{1,2}$, and study its meromorphic continuation. As in the previous section, since we only work at level 2 and for $\alpha$ bounded below, we can fix once and for all a large $\beta>0$ such that all the states consider in this section belong to the weighted space $e^{-\beta c}\cH$.


Now, we make some definitions. For $\boldsymbol{w}\in(\D\cap\H)^r$, $\boldx\in\I^{r_\del}$, we introduce the following element of $e^{-\beta c}\cH$:
\begin{align*}
\Psi_\alpha^\del(\boldw;\boldx)
&:=\underset{\epsilon\to0}\lim\,e^{\frac{1}{2}(\alpha+(2r+r_\del)\gamma-Q)c}\E_\varphi\left[\epsilon^\frac{\alpha^2}{4}e^{\frac{\alpha}{2}X_\epsilon(0)}\prod_{j=1}^r\epsilon^\frac{\gamma^2}{2}e^{\gamma X_\epsilon(w_j)}\prod_{j_\del=1}^{r_\del}\epsilon^\frac{\gamma^2}{4}e^{\frac{\gamma}{2}X_\epsilon(x_{j_\del})}\right.\\
&\qquad\times\left.\exp\left(-\mu e^{\gamma c}\int_{\D\cap\H}\epsilon^\frac{\gamma^2}{2}e^{\gamma X_\epsilon(z)}|\d z|^2-e^{\frac{\gamma}{2}c}\int_\I\epsilon^\frac{\gamma^2}{4}e^{\frac{\gamma}{2}X_\epsilon(x)}\mu_\del(x)\d x\right)\right],
\end{align*}
In practice, we will only be dealing with cases where $2r+r_\del\leq3$. Finally, given Young diagrams $\bolds=(s_1,...,s_r)\in\N^r$ and $\tilde{\bolds}=(\tilde{s}_1,...,\tilde{s}_{\tilde{r}})\in\N^{\tilde{r}}$ we introduce the integrals (provided they are well-defined in $e^{-\beta c}\cH$)
\begin{equation}\label{eq:def_integrals_boundary}
\cI^\del_{\bolds,\tilde{\bolds}}(\alpha):=\int_{\I_-^r}\int_{\I_+^{\tilde r}}\Psi_\alpha^\del(x_1,...,x_r,\tilde{x}_1,...,\tilde{x}_{\tilde{r}})\prod_{j=1}^r\frac{\d x_j}{|x_j|^{s_j}}\prod_{\tilde{j}=1}^{\tilde{r}}\frac{\d\tilde{x}_{\tilde{j}}}{|\tilde{x}_{\tilde{j}}|^{\tilde{s}_j}}.
\end{equation}
In practice, we will only deal with cases where $|\bolds|+|\tilde{\bolds}|\leq2$. In particular, the total number of insertions is bounded by 2. We also define
\begin{equation}\label{eq:def_int_bound_bis}
\cI_{(2)}(\alpha):=\int_{\D\cap\H}\Psi_\alpha^\del(w)\Re(w^{-2})|\d w|^2.
\end{equation}

Since this does not seem to have appeared in the literature yet, we include the following lemma on the analytic continuation of $\Psi_\alpha^\del(\boldw)$, which is a straightforward adaptation of \cite[Theorem 6.1]{KRV_DOZZ}.
\begin{lemma}\label{lem:analytic}
In the setting above, for all $\beta>Q-\alpha-r\gamma$ and $\epsilon>0$, the map $\alpha\mapsto\Psi_\alpha^\del(\boldw)\in e^{-\beta c}\cH$ has an analytic continuation in a complex neighbourhood of $(-\infty,Q-\epsilon)$.
\end{lemma}
\begin{proof}
For notational simplicity and since the question is about the local behaviour of GMC around 0, we only treat the case with no $\gamma$-insertions at the other marked points (i.e. $r=0$). The strategy applies verbatim to the case $r\geq1$.

Let $\alpha\in(-\infty,Q-\epsilon)$ and define the a.s. increasing function of $t\geq0$:
\[Z_t^\alpha:=\mu e^{\gamma c}\int_{\A_t\cap\H}\frac{\d M_\gamma(z)}{|z|^{\gamma\alpha}}+e^{\frac{\gamma}{2}c}\int_{\I\setminus e^{-t}\I}\mu_\del(x)\frac{\d L_\gamma(x)}{|x|^{\frac{\gamma\alpha}{2}}}.\] Under the free field semigroup \eqref{eq:def_semigroups_boundary}, the zero mode evolves as an independent Brownian motion at speed 2: $c_t=c+B_{2t}$. (Compared to the bulk scenario, the speed 2 comes from the fact that the Neumann Green function in $\H$ is $\log\frac{1}{|z-w|}+\log\frac{1}{|z-\bar w|}$, which behaves as \emph{twice} the $\log$ when $w\to0$.) By the Cameron--Martin formula, reweighting by $e^{\frac{\alpha}{2}B_{2t}-\frac{\alpha^2}{4}t}$ amounts to shifting the field by $\alpha\max(\log|\frac1z|,t)$, and we therefore get $\Psi_\alpha^\del=e^{\frac{1}{2}(\alpha-Q)c}\lim_{t\to\infty}\E_\varphi[e^{-Z_t^\alpha}]$ (which is finite for $\alpha<Q$ \cite{GKRV20_bootstrap}). 

Let $F\in\cC$, and for all $t>0$, set $\alpha\mapsto f_t(\alpha):=\E[e^{\frac{\alpha}{2}B_{2t}-\frac{\alpha^2}{4}t}e^{-Z_t^0}F]\in e^{-\beta c}\cH$. The function $f_t$ is analytic in the whole $\alpha$-plane for each $t>0$, hence it suffices to show that $f_t$ converges uniformly on compacts of some neighbourhood of $(-\infty,Q-\epsilon)$, as $t\to\infty$. Arguing as in the proof of Proposition \ref{prop:fusion}, it is sufficient to consider the case $F=1$. Following the proof of \cite[Theorem 6.1]{KRV_DOZZ}, and using the elementary inequality $1-e^{-x}\leq x^p$ for $x>0$ and $p\in(0,1)$, we get:
\[|f_t(\alpha)-f_{t-1}(\alpha)|\leq e^{\Im(\alpha)^2\frac t4}\E[1-e^{-(Z_t^{\Re(\alpha)}-Z_{t-1}^{\Re(\alpha)})}]\leq e^{\Im(\alpha)^2\frac t4}\E[(Z_t^{\Re(\alpha)}-Z_{t-1}^{\Re(\alpha)})^p]=e^{(\frac 14\Im(\alpha)^2-\xi)t}\E[(Z_1^{\Re(\alpha)})^p].\]
Here, $p>0$ is chosen small enough so that the last moment is finite \cite[(3.23)]{huang2018}, and $\xi=\xi(p,\Re(\alpha))>0$ is a multifractal exponent of the type already encountered in the proof of Proposition \ref{prop:fusion}. It is readily seen that $\xi$ is affine and non-increasing in the variable $\Re(\alpha)$, so that $f_t$ converges uniformly in the region $\{\alpha\in\C|\,\Re(\alpha)<Q-\epsilon\text{ and }|\Im(\alpha)|<2\sqrt{\xi(p,Q-\epsilon)}\}$.
\end{proof}

%
%


	\subsection{Expression of the singular state}\label{subsec:expression_boundary}

The first step is to get a probabilistic expression for the singular state.
	
\begin{proposition}\label{prop:expression_boundary}
 For all $\alpha<\alpha_{1,2}$, we have the equality in $e^{-\beta c}\cH$,
\begin{equation}\label{eq:exp_desc_boundary}
\begin{aligned}
\cP_\alpha(\varphi_2)=-\frac{\gamma}{4}\mu_\mathrm{L}\cI^\del_{(2),\emptyset}(\alpha)-\frac{\gamma}{4}\mu_\mathrm{R}\cI^\del_{(2),\emptyset}(\alpha)-\frac{\gamma}{2}\mu\cI_{(2)}(\alpha)+\cR_\alpha,
\end{aligned}
\end{equation}
where $\cR_\alpha$ is analytic in a complex neighbourhood of $\R_-$.
\end{proposition}

\begin{proof}
We proceed as in Proposition \ref{prop:expression_bulk}, except that the modes of the field $(\varphi_n(t))_{n\geq1}$ under the free field dynamics \eqref{eq:def_semigroups_boundary} are now real-valued processes. Under this dynamic, these modes evolve as independent Ornstein--Uhlenbeck processes \cite[Section 7.2]{GRVW}. Moreover, the zero mode evolves as an independent Brownian motion of speed $2$: $c_t=c+B_{2t}$. For all $\varepsilon\in\R$, we introduce the martingale
\[\cE_\varepsilon(t):=e^{e^{2t}(\varepsilon\varphi_2(t)-\frac{\varepsilon^2}{4}\sinh(2t))},\]
with initial value $\cE_\varepsilon(0)=e^{\varepsilon\varphi_2}$. Observe that
\[\frac{\del}{\del\varepsilon}_{|\varepsilon=0}\cE_\varepsilon(t)=e^{2t}\varphi_2(t).\]

By Cameron--Martin's theorem, the effect of reweighting the measure by $e^{-\varepsilon\varphi_2}\cE_\varepsilon(t)$ is to shift the field $X_\D$ as follows
\[X_\D(z)\mapsto X_\D(z)+\frac{\varepsilon}{2}\Re(z^{-2}-\bar{z}^2).\]
Moreover, the reweighting by $e^{\varepsilon\varphi_2-\frac{\varepsilon^2}{8}}$ amounts to the shift $\varphi\mapsto\varphi+\frac{\varepsilon}{2}\Re(e_2)$. Adding the two shifts gives
\[X\mapsto X+\frac{\varepsilon}{2}\Re(z^{-2}).\]

Hence, for all $F\in\cC$ and all $t>0$, we have:
\begin{align*}
&e^{t(2\Delta_\alpha+2)}\E\left[e^{-t\bH}\left(\varphi_2e^{\frac{1}{2}(\alpha-Q)c}\right)F\right]\\
&\quad=e^{\frac{1}{2}(\alpha-Q)c}\frac{\del}{\del\varepsilon}_{|\varepsilon=0}\E\left[e^{\frac{\alpha}{2}B_{2t}-\frac{\alpha^2}{4}t}\cE_\varepsilon(t)e^{-\mu e^{\gamma c}M_\gamma(\A_t^+)}e^{-\mu_\mathrm{L}e^{\frac{\gamma}{2}c}L_\gamma(\I_-)}e^{-\mu_\mathrm{R}e^{\frac{\gamma}{2}c}L_\gamma(\I_+)}F\right]\\
&\quad=e^{\frac{1}{2}(\alpha-Q)c}\frac{\del}{\del\varepsilon}_{|\varepsilon=0}\E\left[e^{\frac{\alpha}{2}B_{2t}-\frac{\alpha^2}{4}t}\exp\left(-\mu e^{\gamma c}\int_{\A_t^+}e^{\gamma\frac{\varepsilon}{2}\Re(z^{-2})}\d M_\gamma(z)-e^{\frac{\gamma}{2}c}\int_{\I_t}e^{\frac{\gamma\varepsilon}{4x^2}}\mu_\del(x)\d L_\gamma(x)\right)F(c,\varphi+\frac{\varepsilon}{2}\cos(2\theta))\right]\\
&\quad=-\frac{\mu\gamma}{2}e^{\frac{1}{2}(\alpha+2\gamma-Q)c}\int_{\A_t^+}\E\left[e^{\frac{\alpha}{2}B_{2t}-\frac{\alpha^2}{4}t}e^{\gamma X(w)-\frac{\gamma^2}{2}\E[X(w)^2]}e^{-\mu e^{\gamma c}M_\gamma(\A_t^+)}F(\varphi)\right]\Re(w^{-2})\frac{|\d w|^2}{(2\Im(w))^\frac{\gamma^2}{2}}\\
&\qquad-\frac{\gamma}{4}e^{\frac{1}{2}(\alpha+\gamma-Q)}\int_{\I_t}\E\left[e^{\frac{\alpha}{2}B_{2t}-\frac{\alpha^2}{4}t}e^{-\mu e^{\gamma c}M_\gamma(\A_t^+)-\mu_\mathrm{L}e^{\frac{\gamma}{2}c}L_\gamma(\I_-)-\mu_\mathrm{R}e^{\frac{\gamma}{2}c}L_\gamma(\I_+)}F\right]\mu_\del(x)\frac{\d x}{|x|^2}\\
&\qquad+\frac{1}{4}\E\left[e^{\frac{\alpha}{2}B_{2t}-\frac{\alpha^2}{4}t}e^{-\mu e^{\gamma c}M_\gamma(\A_t^+)-\mu_\mathrm{L}e^{\frac{\gamma}{2}c}L_\gamma(\I_-)-\mu_\mathrm{R}e^{\frac{\gamma}{2}c}L_\gamma(\I_+)}\del_2F\right].
\end{align*}
Taking the limit as $t\to\infty$, the first (resp. second) line of the last equality gives the first (resp. second and third) terms of \eqref{eq:exp_desc_boundary}. The last line of the last equality converges to $\E[\Psi_\alpha^\del\del_2F]$, which has an analytic continutation in a neighbourhood of $(-\infty,Q)$.
\end{proof}

	\subsection{First pole of $\cP_\alpha$ and the $(1,2)$-HEM}\label{subsec:first_pole_boundary}

As an immediate corollary, we get the $(1,2)$-equation.

\begin{proposition}
We have in $e^{-\beta c}\cH$
\[\underset{\alpha=\alpha_{1,2}}{\mathrm{Res}}\,\cP_\alpha(\varphi_2)=\frac{1}{2}(\mu_\mathrm{L}+\mu_\mathrm{R})\Psi_{\alpha_{-1,2}}^\del.\]
Thus, the $(1,2)$-HEM \eqref{eq:boundary_(1,2)} holds.
\end{proposition}
\begin{proof}
It is easy to see that the last two terms of \eqref{eq:exp_desc_boundary} converge and are analytic in a neighbourhood of $\alpha_{1,2}$. Hence, we only need to treat the boundary integrals (first two terms).

We use the estimate $\Psi_\alpha^\del(x)=|x|^{-\frac{\gamma\alpha}{2}}(\Psi_{\alpha+\gamma}^\del+O(|x|^\xi))$ in $\cC'$, for some $\xi>0$ independent of $\alpha$. This estimate is a straightforward adaptation of Proposition \ref{prop:fusion_boundary} in the case of a single $\gamma$-insertion. From this, we get the following equality in $\cC'$, for $\alpha<\alpha_{1,2}$
\begin{align*}
\int_0^1\Psi_\alpha^\del(x)\frac{\d x}{x^2}
&=\Psi_{\alpha+\gamma}^\del\int_0^1x^{-\frac{\gamma\alpha}{2}-2}\d x+\int_0^1(\Psi_\alpha^\del(x)-|x|^{-\frac{\gamma\alpha}{2}}\Psi_{\alpha+\gamma}^\del)\frac{\d x}{x^2}\\
&=-\frac{2}{\gamma(\alpha-\alpha_{1,2})}\Psi_{\alpha_{1,2}}^\del+O_{\alpha\to\alpha_{1,2}}(1).
\end{align*}

We get the same residue for the term involving the integral over the interval $(-1,0)$. Thus, we get $\underset{\alpha=\alpha_{1,2}}{\mathrm{Res}}\,\cP_\alpha(\varphi_2)=\frac{1}{2}(\mu_\mathrm{L}+\mu_\mathrm{R})\Psi_{\alpha_{-1,2}}$ by combining with Proposition \ref{prop:expression_boundary}. Finally, Lemma \ref{lem:ff_boundary} allows us to conclude that $\Phi_\alpha(\bS_{2Q-\alpha}^0\ind)$ extends to $\alpha_{1,2}$ with the value
\begin{align*}
\Phi_{\alpha_{1,2}}(\bS_{2Q-\alpha_{1,2}}^0\ind)
&=\underset{\alpha\to\alpha_{1,2}}\lim\,2\alpha(\alpha-\alpha_{2,1})(\alpha-\alpha_{1,2})\cP_\alpha(\varphi_2)\\
&=2\alpha_{1,2}(\alpha_{1,2}-\alpha_{2,1})\underset{\alpha=\alpha_{1,2}}{\mathrm{Res}}\,\cP_\alpha(\varphi_2)\\
&=\alpha_{1,2}(\alpha_{1,2}-\alpha_{2,1})(\mu_\mathrm{L}+\mu_\mathrm{R})\Psi_{\alpha_{-1,2}}^\del.
\end{align*}
This concludes the proof since $\alpha_{1,2}(\alpha_{1,2}-\alpha_{2,1})=\frac{4}{\gamma^2}(1-\frac{\gamma^2}{4})$.
\end{proof}

	\subsection{Second pole of $\cP_\alpha$ and the $(2,1)$-HEM}\label{subsec:second_pole_boundary}
	
	In this section, we compute the residue of $\cP_\alpha(\varphi_2)$ at $\alpha=\alpha_{2,1}$, which will prove \eqref{eq:boundary_(2,1)} and end the proof of Theorem \ref{thm:boundary_hem}. It is the content of the following proposition.

\begin{proposition}\label{prop:plan_21_boundary}
We have in $e^{-\beta c}\cH$, for $\alpha<\alpha_{2,1}$,
\begin{equation}\label{eq:mero_boundary}
(\alpha-\alpha_{1,2})\cP_\alpha(\varphi_2)=-\frac{\gamma}{2}(\alpha-\alpha_{1,2})\mu\cI_{(2)}(\alpha)+\frac{\gamma^2}{8}\mu_\mathrm{L}^2\cI_{(1,1),\emptyset}^\del(\alpha)-\frac{\gamma^2}{4}\mu_\mathrm{L}\mu_\mathrm{R}\cI_{(1),(1)}^\del(\alpha)+\frac{\gamma^2}{8}\cI_{\emptyset,(1,1)}^\del(\alpha)+\cR_\alpha,
\end{equation}
where $\cR_\alpha$is analytic up to a neighbourhood of $\alpha_{2,1}$. 

Moreover, for $\gamma<\sqrt{2}$, we have in $e^{-\beta c}\cH$,
\begin{equation}\label{eq:residue_I_two_boundary}
\begin{aligned}
&\underset{\alpha=\alpha_{2,1}}{\mathrm{Res}}\,\cI_{(2)}(\alpha)=-\frac{1}{\gamma}\frac{\frac{\gamma^2}{4}}{1-\frac{\gamma^2}{4}}\sin(\pi\frac{\gamma^2}{4})\frac{\Gamma(\frac{\gamma^2}{4})\Gamma(1-\frac{\gamma^2}{2})}{\Gamma(1-\frac{\gamma^2}{4})}\Psi_{\alpha_{-2,1}}^\del;\\
&\underset{\alpha=\alpha_{2,1}}{\mathrm{Res}}\,\cI^\del_{\emptyset,(1,1)}(\alpha)=\underset{\alpha=\alpha_{2,1}}{\mathrm{Res}}\,\cI^\del_{(1,1),\emptyset}(\alpha)=-\frac{2}{\gamma}\frac{\Gamma(\frac{\gamma^2}{4})\Gamma(1-\frac{\gamma^2}{2})}{\Gamma(1-\frac{\gamma^2}{4})}\Psi_{\alpha_{-2,1}}^\del;\\
&\underset{\alpha=\alpha_{2,1}}{\mathrm{Res}}\,\cI^\del_{(1),(1)}(\alpha)=-\frac{2}{\gamma}\cos(\pi\frac{\gamma^2}{4})\frac{\Gamma(\frac{\gamma^2}{4})\Gamma(1-\frac{\gamma^2}{2})}{\Gamma(1-\frac{\gamma^2}{4})}\Psi_{\alpha_{-2,1}}^\del.
\end{aligned}
\end{equation}
For $\gamma>\sqrt{2}$, all these residues vanish (the integrals are regular at $\alpha_{2,1}$).
\end{proposition}

Assuming this proposition, we can easily conclude the proof of the $(2,1)$-equation.

\begin{proof}[Proof of \eqref{eq:boundary_(2,1)}]
The case $\gamma>\sqrt 2$ is clear from Proposition \ref{prop:plan_21_boundary} since all residues vanish. Hence, we assume $\gamma<\sqrt{2}$. Using successively Lemma \ref{lem:ff_boundary}, \eqref{eq:mero_boundary} and \eqref{eq:residue_I_two_boundary}, we find that $\Phi_\alpha(\bS_{2Q-\alpha}^0\ind)$ extends to $\alpha_{2,1}$ with the value
\begin{align*}
\Phi_{\alpha_{2,1}}(\bS_{2Q-\alpha_{2,1}}^0\ind)
&=\underset{\alpha\to\alpha_{2,1}}\lim\,2\alpha(\alpha-\alpha_{1,2})(\alpha-\alpha_{2,1})\cP_\alpha(\varphi_2)\\
&=2\alpha_{2,1}\underset{\alpha=\alpha_{2,1}}{\mathrm{Res}}\left(-\mu\frac{\gamma}{2}(\alpha_{2,1}-\alpha_{1,2})\cI_{(2)}(\alpha)+\mu_\mathrm{L}^2\frac{\gamma^2}{8}\cI_{(1,1),\emptyset}^\del(\alpha)-\mu_\mathrm{R}\mu_\mathrm{R}\frac{\gamma^2}{4}\cI_{(1),(1)}^\del(\alpha)+\mu_\mathrm{R}^2\frac{\gamma^2}{8}\cI_{\emptyset,(1,1)}^\del(\alpha)\right)\\
&=\frac{\gamma^2}{4}\left(\mu_\mathrm{L}^2-2\mu_\mathrm{L}\mu_\mathrm{R}\cos(\pi\frac{\gamma^2}{4})+\mu^2_\mathrm{R}-\mu\sin(\pi\frac{\gamma^2}{4})\right)\frac{\Gamma(\frac{\gamma^2}{4})\Gamma(1-\frac{\gamma^2}{2})}{\Gamma(1-\frac{\gamma^2}{4})}.
\end{align*}
\end{proof}

The remainder of this section is devoted to the proof of Proposition \ref{prop:plan_21_boundary}. The strategy is the same as in the bulk case. To prove \eqref{eq:residue_I_two_boundary}, we use fusion estimates (Proposition \ref{prop:fusion_boundary}) to factorise the residue into the product of a primary field and a Selberg integral. The residue of the Selberg integral can be evaluated explicitly (Lemma \ref{lem:ff_residues}). The analytic continuation of the Poisson operator of \eqref{eq:mero_boundary} is handled using integration by parts and a derivative formula for the boundary states.

		\subsubsection{Residues of Selberg integrals}
	
We refer to Appendix \ref{app:selberg} for the definition of the Selberg integrals $S_{2,2}$ and $S_{2,1}$. By \eqref{eq:value_selberg}, the value of $S_{2,2}(1,-\frac{\gamma\alpha}{2},-\frac{\gamma^2}{4})$ is
\[S_{2,2}\left(1,-\frac{\gamma\alpha}{2},-\frac{\gamma^2}{4}\right)=-\frac{2}{\gamma(\alpha-\alpha_{2,1})}\frac{\Gamma(-\frac{\gamma\alpha}{2})\Gamma(1-\frac{\gamma^2}{2})}{\Gamma(1-\frac{\gamma\alpha}{2}-\frac{\gamma^2}{2})},\]
so that $S_{2,2}(1,-\frac{\gamma\alpha}{2},-\frac{\gamma^2}{4})$ has a simple pole at $\alpha_{2,1}$, with residue
\begin{equation}\label{eq:def_R}
\underset{\alpha=\alpha_{2,1}}{\mathrm{Res}}\,S_{2,2}\left(1,-\frac{\gamma\alpha}{2},-\frac{\gamma^2}{4}\right)=-\frac{2}{\gamma}\frac{\Gamma(\frac{\gamma^2}{4})\Gamma(1-\frac{\gamma^2}{2})}{\Gamma(1-\frac{\gamma^2}{4})}.
\end{equation}

The next lemma evaluates the residues of related integrals at $\alpha_{2,1}$.
\begin{lemma}\label{lem:ff_residues}
Suppose $\gamma<\sqrt{2}$. We have
\begin{align*}
&\underset{\alpha\nearrow\alpha_{2,1}}\lim\,(\alpha-\alpha_{2,1})\int_{\D\cap\H}|w|^{-\gamma\alpha}|w-\bar{w}|^{-\frac{\gamma^2}{2}}\Re(w^{-2})|\d w|^2=-\frac{1}{\gamma}\frac{\frac{\gamma^2}{4}}{1-\frac{\gamma^2}{4}}\sin(\pi\frac{\gamma^2}{4})\frac{\Gamma(\frac{\gamma^2}{4})\Gamma(1-\frac{\gamma^2}{2})}{\Gamma(1-\frac{\gamma^2}{4})};\\
&\underset{\alpha\nearrow\alpha_{2,1}}\lim\,(\alpha-\alpha_{2,1})\int_{-1}^0\int_0^1|x_1|^{-\frac{\gamma\alpha}{2}-1}|x_2|^{-\frac{\gamma\alpha}{2}-1}|x_1-x_2|^{-\frac{\gamma^2}{2}}\d x_1\d x_2=-\frac{2}{\gamma}\cos(\pi\frac{\gamma^2}{4})\frac{\Gamma(\frac{\gamma^2}{4})\Gamma(1-\frac{\gamma^2}{2})}{\Gamma(1-\frac{\gamma^2}{4})}.
\end{align*}
\end{lemma}
\begin{proof}
\emph{First identity.} 

Writing the integral in polar variables, we have
\begin{align*}
\int_{\D\cap\H}|w|^{-\gamma\alpha}|w-\bar{w}|^{-\frac{\gamma^2}{2}}\Re(w^{-2})|\d w|^2
&=\int_0^1r^{-\gamma\alpha-\frac{\gamma^2}{2}-1}\d r\int_0^\pi(2\sin\theta)^{-\frac{\gamma^2}{2}}\cos(2\theta)\d\theta\\
&=-\frac{1}{\gamma(\alpha-\alpha_{2,1})}\int_0^\pi(2\sin(\theta))^{-\frac{\gamma^2}{2}}(2\cos(\theta)^2-1)\d\theta.
\end{align*}
 We recall the formula $B(a,b)=2\int_0^\frac{\pi}{2}\sin(\theta)^{2a-1}\cos(\theta)^{2b-1}\d\theta$ for the Beta function, and the expression $B(a,b)=\frac{\Gamma(a)\Gamma(b)}{\Gamma(a+b)}$. Using the identity $\Gamma(z+1)=z\Gamma(z)$, Euler's formula $\Gamma(z)\Gamma(1-z)=\frac{\pi}{\sin(\pi z)}$, the duplication formula $\Gamma(z)\Gamma(z+\frac{1}{2})=\sqrt{\pi}2^{1-2z}\Gamma(2z)$, and the value $\Gamma(\frac{1}{2})=\sqrt{\pi}$, the integral in $\theta$ equals
\begin{align*}
2^{2-\frac{\gamma^2}{2}}\int_0^\frac{\pi}{2}\sin(\theta)^{-\frac{\gamma^2}{2}}\cos(\theta)^2\d\theta-2^{1-\frac{\gamma^2}{2}}\int_0^\frac{\pi}{2}\sin(\theta)^{-\frac{\gamma^2}{2}}\d\theta
&=2^{1-\frac{\gamma^2}{2}}B(\frac{1}{2}-\frac{\gamma^2}{4},\frac{3}{2})-2^{-\frac{\gamma^2}{2}}B(\frac{1}{2}-\frac{\gamma^2}{4},\frac{1}{2})\\
&=2^{-\frac{\gamma^2}{2}}\sqrt{\pi}\Gamma(\frac{1}{2}-\frac{\gamma^2}{4})\left(\frac{1}{\Gamma(2-\frac{\gamma^2}{4})}-\frac{1}{\Gamma(1-\frac{\gamma^2}{4})}\right)\\
&=-\frac{\Gamma(1-\frac{\gamma^2}{2})}{\Gamma(1-\frac{\gamma^2}{4})}\sin(\pi\frac{\gamma^2}{4})\left(\Gamma(\frac{\gamma^2}{4}-1)+\Gamma(\frac{\gamma^2}{4})\right)\\
&=\frac{\frac{\gamma^2}{4}}{1-\frac{\gamma^2}{4}}\frac{\Gamma(1-\frac{\gamma^2}{2})\Gamma(\frac{\gamma^2}{4})}{\Gamma(1-\frac{\gamma^2}{4})}\sin(\pi\frac{\gamma^2}{4}).
\end{align*}

\emph{Second identity.}

 Let us consider the function
\begin{align*}
&I(\alpha):=\int_{-1}^0\int_0^1|x_1|^{-\frac{\gamma\alpha}{2}-1}|x_2|^{-\frac{\gamma\alpha}{2}-1}|x_2-x_1|^{-\frac{\gamma^2}{2}}|x_1-1|^{-1+\frac{\gamma^2}{2}+\frac{\gamma\alpha}{2}}|x_2-1|^{-1+\frac{\gamma^2}{2}+\frac{\gamma\alpha}{2}}\d x_1\d x_2,
\end{align*}
which is well-defined and analytic for $\Re(\alpha)\in(-\gamma,-\frac{\gamma}{2})$. Using the change of variable $x_j=\frac{t_j-1}{t_j+1}$ with Jacobian $\d x_j=\frac{2\d t_j}{(t_j+1)^2}$, $j=1,2$, we have by \eqref{eq:selberg}:

\begin{align*}
I(\alpha)
&=2^{\frac{\gamma^2}{2}+\gamma\alpha}\int_0^1\int_1^\infty|1-t_1|^{-\frac{\gamma\alpha}{2}-1}|1-t_2|^{-\frac{\gamma\alpha}{2}-1}|t_1-t_2|^{-\frac{\gamma^2}{2}}\d t_1\d t_2\\
&=2^{\frac{\gamma^2}{2}+\gamma\alpha}S_{2,1}(1,-\frac{\gamma\alpha}{2},-\frac{\gamma^2}{4})\\
&=2^{\frac{\gamma^2}{2}+\gamma\alpha}\cos(\pi\frac{\gamma^2}{4})\frac{\sin(\pi\frac{\gamma^2}{4})}{\sin\pi(\frac{\gamma\alpha}{2}+\frac{\gamma^2}{2})} S_{2,2}(1,-\frac{\gamma\alpha}{2},-\frac{\gamma^2}{4}).
\end{align*}

Since $(x_1,x_2)\mapsto|x_1-1|^{-1+\frac{\gamma^2}{2}+\frac{\gamma\alpha}{2}}|x_2-1|^{-1+\frac{\gamma^2}{2}+\frac{\gamma\alpha}{2}}$ is smooth and converges to 1 as $x_1,x_2\to0$, we have
\begin{align*}
&\underset{\alpha\to\alpha_{2,1}}\lim\,(\alpha-\alpha_{2,1})\int_{-1}^0\int_0^1|x_1|^{-\frac{\gamma\alpha}{2}-1}|x_2|^{-\frac{\gamma\alpha}{2}-1}|x_1-x_2|^{-\frac{\gamma^2}{2}}\d x_1\d x_2\\
&\quad=\underset{\alpha\to\alpha_{2,1}}\lim\,(\alpha-\alpha_{2,1})I(\alpha)\\
&\quad=2^{\frac{\gamma^2}{2}+\gamma\alpha_{2,1}}\cos(\pi\frac{\gamma^2}{4})\frac{\sin(\pi\frac{\gamma^2}{4})}{\sin\pi(\frac{\gamma\alpha_{2,1}}{2}+\frac{\gamma^2}{2})}\underset{\alpha\to\alpha_{2,1}}\lim\,(\alpha-\alpha_{2,1})S_{2,2}(-\frac{\gamma\alpha}{2},1,-\frac{\gamma^2}{4})\\
&\quad=-\frac{2}{\gamma}\cos(\pi\frac{\gamma^2}{4})\frac{\Gamma(\frac{\gamma^2}{4})\Gamma(1-\frac{\gamma^2}{2})}{\Gamma(1-\frac{\gamma^2}{4})}.
\end{align*}
\end{proof}

		\subsubsection{Fusion estimates}
From here, we are in position to conclude the computation of $\underset{\alpha=\alpha_{2,1}}{\mathrm{Res}}\,\cP_\alpha(\varphi_2)$. We will rely on the following fusion estimates.

\begin{proposition}\label{prop:fusion_boundary}
The following estimates hold in $\cC'$:
\begin{enumerate}[label={\arabic*.}]
\item \label{item:geq_boundary} Suppose $\alpha+2\gamma\geq Q$. Then,
\begin{align*}
&\Psi_\alpha^\del(w)=|w|^{-\gamma\alpha}O(\Im(w)^{\frac{1}{2}(\alpha+2\gamma-Q)^2});\\
&\Psi_\alpha^\del(x_1,x_2)=|x_1|^{-\frac{\gamma\alpha}{2}}|x_2|^{-\frac{\gamma\alpha}{2}}|x_1-x_2|^{-\frac{\gamma^2}{2}}O((|x_1|\vee|x_2|)^{\frac{1}{4}(\alpha+2\gamma-Q)^2}).
\end{align*}
\item \label{item:less_boundary} Suppose $\alpha+2\gamma<Q$ and define $\Delta:=\{(x_1,x_2)\in(\I\setminus\{0\})^2|\,|\log|x_2/x_1||\geq1\}$. There exists $\xi>0$ such that:
\begin{align*}
&\Psi_\alpha^\del(w)=|w|^{-\gamma\alpha}(2\Im(w))^{-\frac{\gamma^2}{2}}(\Psi_{\alpha+2\gamma}^\del+O(|w|^\xi)),&&\text{uniformly over }\D\cap\H;\\
&\Psi_\alpha^\del(x_1,x_2)=|x_1|^{-\frac{\gamma\alpha}{2}}|x_2|^{-\frac{\gamma\alpha}{2}}|x_1-x_2|^{-\frac{\gamma^2}{2}}\left(\Psi_{\alpha+2\gamma}^\del+O((|x_1|\vee|x_2|)^\xi)\right),&&\text{uniformly over }\Delta.
\end{align*}
\end{enumerate}
\end{proposition}
The proof of this proposition is identical to that of \ref{prop:fusion}, and we omit it. Using these estimates, we can compute the residues of $\cI_{(2)}$, $\cI_{\emptyset,(1,1)}$, $\cI_{(1),(1)}$, $\cI_{(1,1),\emptyset}$ at $\alpha_{2,1}$.

\begin{proof}[Proof of \eqref{eq:residue_I_two_boundary}]
\emph{Case $\gamma<\sqrt{2}$}. 

In this case, we have $\alpha_{-2,1}=\alpha_{2,1}+2\gamma<Q$. We write for $\alpha<\alpha_{2,1}$.
\begin{align*}
\cI_{(2)}(\alpha)
&=\Psi_{\alpha+2\gamma}^\del\int_{\D\cap\H}|w|^{-\gamma\alpha}(2\Im(w))^{-\frac{\gamma^2}{2}}\Re(w^{-2})|\d w|^2\\
&\quad+\int_{\D\cap\H}(\Psi_\alpha^\del(w)-|w|^{-\gamma\alpha}(2\Im(w))^{-\frac{\gamma^2}{2}}\Psi_{\alpha+2\gamma}^\del)\Re(w^{-2})|\d w|^2.
\end{align*}
By item \ref{item:less_boundary} of Proposition \ref{prop:fusion_boundary}, the last line is absolutely convergent and analytic in a neighbourhood of $\alpha_{2,1}$. By Lemma \ref{lem:ff_residues}, we deduce the following limit in $\cC'$:
\[\underset{\alpha\to\alpha_{2,1}}\lim\,(\alpha-\alpha_{2,1})\cI_{(2)}(\alpha)=-\frac{1}{\gamma}\frac{\frac{\gamma^2}{4}}{1-\frac{\gamma^2}{4}}\frac{\Gamma(\frac{\gamma^2}{4})\Gamma(1-\frac{\gamma^2}{2})}{\Gamma(1-\frac{\gamma^2}{4})}\Psi_{\alpha_{-2,1}}^\del.\]
Reproducing the argument at the end of the proof of Proposition \ref{prop:first_pole}, we deduce that this equality actually holds in $e^{-\beta c}\cH$, which give the first line of \eqref{eq:residue_I_two_boundary}. The proof of the last two lines of \eqref{eq:residue_I_two_boundary} is identical, with the only difference that we must first remove the region where $\log|x_2/x_1|<1$ where the fusion estimate of the second item of Proposition \ref{prop:fusion_boundary} does not hold. This strategy has already been carried out in the proof of \eqref{eq:residue_I_two}, and we omit the details.

\emph{Case $\gamma>\sqrt{2}$}.

In this case, we have $\alpha_{2,1}+2\gamma>Q$. We only show that $\cI_{\emptyset,(1,1)}$ is regular at $\alpha_{2,1}$, the other case being dealt with similarly. We consider the following regions of $(0,1)^2$:
\[\cD_0:=\{x_1\vee x_2\leq e(x_1\wedge x_2)\};\qquad\cD_1:=\{x_1\leq e^{-1}x_2\}.\]
We use the notation $\I_t:=(e^{-t},1)$. For $r\in(0,1)$, we have by scaling and the first item of Proposition \ref{prop:fusion_boundary}
\[\int_{r\I_2^2}\Psi_\alpha^\del(x_1,x_2)\frac{\d x_1}{x_1}\frac{\d x_2}{x_2}=O(r^{-\gamma\alpha-\frac{\gamma^2}{2}+\frac{1}{2}(Q-\alpha-2\gamma)^2})\qquad\text{in }\cC'\]
For $\alpha=\alpha_{2,1}$, the exponent equals $\frac{1}{2}(\frac{2}{\gamma}-\gamma)^2>0$. Hence, $\int_{\cD_0}\Psi_\alpha^\del(x_1,x_2)\frac{\d x_1}{x_1}\frac{\d x_2}{x_2}$ is bounded by an absolutely convergent series in a neighbourhood of $\alpha_{2,1}$, so the integral converges and is analytic in this region.

In the region $\cD_1$, we have $|x_1-x_2|^{-\frac{\gamma^2}{2}}\leq(1-e^{-1})x_2)^{-\frac{\gamma^2}{2}}$. Proposition \ref{prop:fusion_boundary} then gives $\Psi_\alpha^\del(x_1,x_2)=O(|x_1|^{-\frac{\gamma\alpha}{2}}|x_2|^{-\frac{\gamma^2}{2}-\frac{\gamma\alpha}{2}})$, which is uniformly integrable on $(\cD_1,\frac{\d x_1}{x_1}\frac{\d x_2}{x_2})$ in a neighbourhood of $\alpha_{2,1}$. This concludes the proof of analyticity of $\cI_{\emptyset,(1,1)}$ around $\alpha_{2,1}$.
\end{proof}

		\subsubsection{Meromorphic continuation}

Similar to the bulk version, we have the following derivative formula: for all $F\in\cC$,
\begin{equation}\label{eq:derivative_formula_boundary}
\begin{aligned}
\del_x\E[\Psi_\alpha^\del(x)F]
&=-\frac{\alpha\gamma}{2x}\E[\Psi_\alpha^\del(x)F]-\mu\frac{\gamma^2}{2}\int_{\D\cap\H}\E[\Psi_\alpha^\del(w;x)F]\del_xG(x,w)|\d w|^2\\
&\quad-\mu_\mathrm{L}\frac{\gamma^2}{4}\int_{-1}^0\E[\Psi_\alpha^\del(x,x')F]\del_xG(x,x')\d x'-\mu_\mathrm{R}\frac{\gamma^2}{4}\int_0^1\E[\Psi_\alpha^\del(x,x')F]\del_xG(x,x')\d x'\\
&\quad+\E[\Psi_\alpha^\del(x)\nabla F(\del G_\del(x,\cdot\,))].
\end{aligned}
\end{equation}
This formula is only valid provided all the terms involved are absolutely convergent, which is the case for $\alpha$ in a complex neighbourhood of $-\infty$. It is then extended by analytic continuation to the domain of analyticity of the RHS.		
		
Now, we express the analytic continuation of $\cP_\alpha(\varphi_2)$ up to $\alpha_{2,1}$. As in the bulk case, the proof relies on a combination of integration by parts and the derivative formula. Fortunately, it happens to be much less tedious.

\begin{proposition}\label{prop:mero_boundary}
For all $\alpha<\alpha_{1,2}$, we have
\begin{align*}
&(\alpha-\alpha_{1,2})\cI_{\emptyset,(2)}^\del(\alpha)=-\frac{\gamma}{2}\mu_\mathrm{R}\cI^\del_{\emptyset,(1,1)}(\alpha)+\frac{\gamma}{2}\mu_\mathrm{L}\cI_{(1),(1)}^\del(\alpha)+\cR_\alpha;\\
&(\alpha-\alpha_{1,2})\cI_{(2),\emptyset}^\del(\alpha)=-\frac{\gamma}{2}\mu_\mathrm{L}\cI^\del_{(1,1),\emptyset}(\alpha)+\frac{\gamma}{2}\mu_\mathrm{R}\cI^\del_{(1),(1)}(\alpha)+\cR'_\alpha,
\end{align*}
where $\cR_\alpha$, $\cR'_\alpha$ are regular at $\alpha_{2,1}$.
\end{proposition}
\begin{proof}
By integration by parts, we have for all $\alpha<\alpha_{1,2}$ and all $F\in\cC$:
\[\int_0^1\E[\Psi_\alpha^\del(x)F]\frac{\d x}{x^2}=\int_0^1\del_x\E[\Psi_\alpha^\del(x)F]\frac{\d x}{x}-\E[\Psi_\alpha^\del(1)F].\]
As usual, the last term is interpreted using the Girsanov transform. As in the proof of Proposition \ref{prop:mero} (and \cite[Section 3.3]{BW}), the validity of this formula is for $\alpha<\alpha_{1,2}$ where we have absolute convergence. The formula is then extended to the domain of analyticity of the RHS. Combining with \eqref{eq:derivative_formula_boundary} gives
\begin{align*}
\frac{\gamma}{2}(\alpha-\alpha_{1,2})\cI_{\emptyset,(2)}
&=-\mu_\mathrm{L}\frac{\gamma^2}{4}\int_{-1}^0\int_0^1\E[\Psi_\alpha^\del(x_1,x_2)F]\del_{x_1}G(x_1,x_2)\frac{\d x_1}{x_1}\d x_2\\
&\quad-\mu_\mathrm{R}\frac{\gamma^2}{4}\int_0^1\int_0^1\E[\Psi_\alpha^\del(x_1,x_2)F]\del_{x_1}G(x_1,x_2)\frac{\d x_1}{x_1}\d x_2\\
&\quad-\mu\frac{\gamma^2}{2}\int_{\D\cap\H}\int_0^1\E[\Psi_\alpha^\del(w;x_1)F]\del_{x_1}G(x_1,w)|\d w|^2\frac{\d x_1}{x_1}\\
&\quad+\int_0^1\E[\Psi_\alpha^\del(x)\nabla F(\del_{x_1}G_\del(x_1,\cdot\,))]\frac{\d x_1}{x_1}-\E[\Psi_\alpha^\del(1)F]\\
&=\mu_\mathrm{L}\frac{\gamma^2}{4}\cI_{(1),(1)}^\del(\alpha)-\mu_\mathrm{R}\frac{\gamma^2}{4}\cI_{\emptyset,(1,1)}^\del(\alpha)+\cR_\alpha.
\end{align*}
In the last line, we have defined $\cR_\alpha$ to be the last two lines of the RHS, which are easily seen to converge and be analytic in a neighbourhood of $\alpha_{2,1}$. For the first two lines, we have symmetrised the singularity $\frac{1}{x_1}\frac{1}{x_2-x_1}$ in order to get the expressions $-\cI_{(1),(1)}^\del$ and $\cI_{\emptyset,(1,1)}^\del$. The previous equation is valid for $\alpha<\alpha_{1,2}$, but the RHS is analytic up to $\alpha_{2,1}$, so it expresses the meromorphic extension of $\cI_{\emptyset,(2)}^\del$ in this region. The proof is identical for $\cI^\del_{(2),\emptyset}$.
\end{proof}

\appendix

\section{Selberg \& Dotsenko--Fateev integrals}\label{app:selberg}
We consider the following Selberg integrals, with the notation borrowed from \cite[Equation (2.31)]{ForresterWarnaar}:
\begin{align*}
&S_{2,2}(a,b,c):=\int_0^1\int_0^1|t_1|^{a-1}|t_2|^{a-1}|1-t_1|^{b-1}|1-t_2|^{b-1}|t_2-t_1|^{2c}\d t_1\d t_2\\
&S_{2,1}(a,b,c):=\int_0^1\int_1^\infty|t_1|^{a-1}|t_2|^{a-1}|1-t_1|^{b-1}|1-t_2|^{b-1}|t_1-t_2|^{2c}\d t_1\d t_2.
\end{align*}
The first integral converges for $\Re(a)>0$, $\Re(b)>0$, and $\Re(c)>-\min\{\frac{1}{2},\Re(a),\Re(b)\}$. The second integral converges for $\Re(a)>0$, $\Re(a+b+2c)<1$, and $\Re(c)>-\frac{1}{2}$. Note that $S_{2,2}(a,b,c)=S_{2,2}(b,a,c)$ by the change of variable $t_j\mapsto1-t_j$, $j=1,2$. The integral $S_{2,2}$ extends meromorphically to $\C^3$ via the formula \cite[Equation (1.1)]{ForresterWarnaar}:
\begin{equation}\label{eq:value_selberg}
S_{2,2}(a,b,c)=\frac{\Gamma(a)\Gamma(b)\Gamma(a+c)\Gamma(b+c)\Gamma(1+2c)}{\Gamma(a+b+c)\Gamma(a+b+2c)\Gamma(1+c)}.
\end{equation}
According to \cite[Equation (2.33)]{ForresterWarnaar}, we have
\begin{equation}\label{eq:selberg}
S_{2,1}(a,b,c)=\cos(\pi c)\frac{\sin\pi(a+c)}{\sin\pi(a+b+2c)}S_{2,2}(a,b,c).
\end{equation}

The \emph{Dotsenko--Fateev integral} is a version of the Selberg integral where the domain of integration is the complex plane. Neretin introduced a generalisation of the Dotsenko--Fateev integral \cite{Neretin22}. For a pair of complex numbers $\boldsymbol{a}=(a,\tilde{a})$ such that $a-\tilde{a}\in\Z$, we write $z^{\boldsymbol{a}}=z^a\bar{z}^{\tilde{a}}=|z|^{a+\tilde{a}}e^{i(a-\tilde{a})\arg z}$. Then, for such a triple $\boldsymbol{a},\boldsymbol{b},\boldsymbol{c}$, we consider
\begin{align*}
N(\boldsymbol{a},\boldsymbol{b},\boldsymbol{c}):=\int_{\C^2}w_1^{\boldsymbol{a}-1}w_2^{\boldsymbol{a}-1}(1-w_1)^{\boldsymbol{b}-1}(1-w_2)^{\boldsymbol{b}-1}(w_2-w_1)^{\boldsymbol{2c}}|\d w_1|^2|\d w_2|^2.
\end{align*}
Neretin found an exact formula for the meromorphic extension of this integral \cite[Corollary 1.3]{Neretin22}:
\begin{equation}\label{eq:neretin}
N(\boldsymbol{a},\boldsymbol{b},\boldsymbol{c})=(-1)^{\boldsymbol{c}}S_{2,2}(a,b,c)S_{2,2}(\tilde{a},\tilde{b},\tilde{c})\frac{\sin(\pi a)\sin(\pi b)\sin\pi(a+c)\sin\pi(b+c)\sin\pi(1+2c)}{\sin\pi(a+b+c)\sin\pi(a+b+2c)\sin\pi(1+c)}.
\end{equation}

\bibliographystyle{alpha}
\bibliography{bpz}
\end{document}